\documentclass[reqno]{amsart}
\usepackage{amsmath,amsthm,amscd,amssymb,amsfonts, amsbsy}
\usepackage{latexsym}
\usepackage{color}
\usepackage{enumerate}
\usepackage{pxfonts}
\usepackage{verbatim}

\theoremstyle{plain}
\newtheorem{theorem}[equation]{Theorem}
\newtheorem{lemma}[equation]{Lemma}

\newtheorem{proposition}[equation]{Proposition}

\theoremstyle{definition}
\newtheorem{definition}[equation]{Definition}
\newtheorem{condition}[equation]{Condition}

\theoremstyle{remark}
\newtheorem{remark}[equation]{Remark}

\newtheorem{acknowledgment}[equation]{Acknowledgment}

\newcommand{\dv}{\operatorname{div}}

\newcommand{\diam}{\operatorname{diam}}

\newcommand{\tr}{\operatorname{tr}}

\newcommand*{\tran}{^{\mkern-1.5mu\mathsf{T}}}

\numberwithin{equation}{section}

\newcommand{\bR}{\mathbb{R}}
\newcommand{\bZ}{\mathbb{Z}}

\newcommand\cD{\mathcal{D}}

\newcommand\cL{\mathcal{L}}

\providecommand{\set}[1]{\{#1\}}

\providecommand{\abs}[1]{\lvert#1\rvert}
\providecommand{\Abs}[1]{\left\lvert#1\right\rvert}

\providecommand{\norm}[1]{\lVert#1\rVert}

\renewcommand{\vec}[1]{\boldsymbol{#1}}

\begin{document}
\title[Estimates for linear elliptic operators II]
{On $C^1$, $C^2$, and weak type-$(1,1)$ estimates for linear elliptic operators: Part II}

\author[H. Dong]{Hongjie Dong}
\address[H. Dong]{Division of Applied Mathematics, Brown University,
182 George Street, Providence, RI 02912, United States of America}
\email{Hongjie\_Dong@brown.edu}
\thanks{H. Dong was partially supported by the NSF under agreement DMS-1056737 and DMS-1600593.}

\author[L. Escauriaza]{Luis Escauriaza}
\address[L. Escauriaza]{UPV/EHU, Dpto. Matem\'aticas, Barrio Sarriena s/n 48940 Leioa, Spain}
\email{luis.escauriaza@ehu.eus}
\thanks{L. Escauriaza is supported by grants MTM2014-53145-P and IT641-13 (GIC12/96).}

\author[S. Kim]{Seick Kim}
\address[S. Kim]{Department of Mathematics, Yonsei University, 50 Yonsei-ro, Seodaemun-gu, Seoul 03722, Republic of Korea}
\email{kimseick@yonsei.ac.kr}
\thanks{S. Kim is partially supported by NRF Grant No. NRF-2016R1D1A1B03931680.}



\begin{abstract} We extend and improve the results in \cite{DK16}:  showing that weak solutions to full elliptic equations in divergence form with zero Dirichlet boundary conditions are continuously differentiable up to the boundary when the leading coefficients have Dini mean oscillation and the lower order coefficients verify certain conditions.  Similar results are obtained for non-divergence form equations. We extend the weak type-(1, 1) estimates in \cite{DK16} and  \cite{Es94} up to the boundary and derive a Harnack inequality for non-negative adjoint solutions to non-divergence form elliptic equations, when the leading coefficients have Dini mean oscillation.
\end{abstract}
\maketitle

\section{Introduction and main results}
Let $\Omega \subset \bR^n$ be a bounded domain.
We consider a second-order elliptic operator $L$ in divergence form
\begin{equation}						\label{master-d}
L u= \sum_{i,j=1}^n D_i(a^{ij}(x) D_j u + b^i(x) u) + \sum_{i=1}^n c^i(x) D_iu + d(x)u,
\end{equation}
where the coefficients $\mathbf{A}=(a^{ij})_{i,j=1}^n$, $\vec b=(b^1,\ldots, b^n)$, $\vec c=(c^1,\ldots, c^n)$, and $d$ are measurable functions defined on $\overline{\Omega}$.
We assume that the principal coefficients $\mathbf{A}=(a^{ij})$ are defined on $\bR^n$ and satisfy the uniform ellipticity condition
\begin{equation}					\label{ellipticity-d}
\lambda \abs{\xi}^2 \le \sum_{i,j=1}^n a^{ij}(x) \xi^i \xi^j,\quad \forall \xi=(\xi^1,\ldots, \xi^n) \in \bR^n,\quad\forall x \in \bR^n
\end{equation}
and the uniform boundedness condition
\begin{equation}					\label{bdd-d}
\sum_{i,j=1}^n \,\abs{a^{ij}(x)}^2 \le \Lambda^2,\quad \forall x \in \bR^n
\end{equation}
for some positive constants $\lambda$ and  $\Lambda$.

We say that a nonnegative measurable function $\omega: (0,a) \to \bR$ is a Dini function provided that there are constants $c_1, c_2 >0$ such that
\begin{equation}				\label{eq10.31tu}
c_1 \omega(t) \le \omega(s) \le c_2 \omega(t)
\end{equation}
whenever $\tfrac12 t \le s \le  t$ and $0<t<a$ and that
\[
\int_0^t \frac{\omega(s)}s \,ds <+\infty,\quad \forall t \in (0,a).
\]
For $x\in \bR^n$ and $r>0$, we denote by $B(x,r)$ the Euclidean ball with radius $r$ centered at $x$, and denote
\[
\Omega(x,r):=\Omega \cap B(x,r).
\]
For a locally integrable function $g$ on $\Omega$, we shall say that $g$ is uniformly Dini continuous (in $\Omega$) if the function $\varrho_g: \bR_+ \to \bR$ defined by
\[
\varrho_g(r):=\sup_{x \in \Omega} \,\sup_{y, y' \in \Omega(x,r)} \abs{g(y)-g(y')}
\]
is a Dini function, while we shall say that $g$ is of \emph{Dini mean oscillation}
(in $\Omega$) if the function $\omega_g: \bR_+ \to \bR$ defined by
\begin{equation*}					
\omega_g(r):=\sup_{x\in \overline{\Omega}} \fint_{\Omega(x,r)} \,\abs{g(y)-\bar {g}_{\Omega(x,r)}}\,dy \quad \left(\;\bar g_{\Omega(x,r)} :=\fint_{\Omega(x,r)} g\;\right)
\end{equation*}
is a Dini function.
We point out that the condition \eqref{eq10.31tu} is satisfied by $\rho_g$ and also by $\omega_g$; see \cite{Y.Li2016}.
Moreover, it should be clear that if $g$ is uniformly Dini continuous, then it is of Dini mean oscillation and $\omega_g(r) \le \varrho_g(r)$.
It is worthwhile to note that if $\Omega$ is such that for any $x\in \overline{\Omega}$,
\begin{equation}					\label{cond_a}
\abs{\Omega(x,r)} \ge A_0 r^n,\quad 0< \forall r\le \diam \Omega\quad (\text{$A_0$ is a positive constant})
\end{equation}
and if $g$ is of Dini mean oscillation, then $g$ is uniformly continuous with a modulus of continuity determined by $\omega_g$.

In a recent paper \cite{Y.Li2016}, Yanyan Li raised a question whether weak solutions of
\[
\sum_{i,j=1}^n D_i(a^{ij}(x) D_j u)=0
\]
are continuously differentiable when $\mathbf{A}=(a^{ij})$ are of Dini mean oscillation.\footnote{In fact, the condition on $\mathbf{A}$ imposed by Yanyan Li was slightly stronger.}
In \cite{DK16}, the first and third named authors showed that the answer to his question is positive.
This paper is a sequel to \cite{DK16} and extends and improves results presented there.
More precisely, we show that weak solutions to \eqref{master-d} with zero Dirichlet boundary conditions are continuously differentiable up to boundary provided that the leading coefficients $\mathbf{A}$ and $\vec b$ are of Dini mean oscillation, lower order coefficients $\vec c$ and $d$ belong to $L^q$ with $q>n$, and $\partial \Omega$ has $C^{1,Dini}$ boundary.
We prove a similar result when the operator is in non-divergence form.
In \cite{Es94}, the second named author investigated (interior) weak type-$(1,1)$ estimates for solutions of
\[
\sum_{i,j=1}^n a^{ij}(x) D_{ij} u=f\;\mbox{ in }\; B(0,1),\quad u=0\;\mbox{ on }\;\partial B(0,1),
\]
and showed that if $\mathbf{A}=(a^{ij})$ belong to the class of functions with vanishing mean oscillations (VMO), then the $D^2 u$ satisfies weak type-$(1,1)$ estimates with respect to $W\,dx$.
Here $W$ is a nonnegative solution to the adjoint equation, which is a good Muckenhoupt weight as $\log W$ was proved to be in VMO, so that the associated measure $W\,dx$ is better adjusted to the equation than $dx$. Moreover, it is also shown in \cite{Es94} that the standard weak type-$(1,1)$ estimates (i.e., the estimate with $W=1$) do not hold even if $\mathbf{A}$ is uniformly continuous.
In this paper, we prove that if $\mathbf{A}$ is of Dini mean oscillation, then the standard weak type-$(1,1)$ estimates hold up to the boundary.
We also show that in this case, the weight $W$ mentioned above satisfies a Harnack type inequality.

Now, we state the main results more precisely.
\begin{definition}			
Let $\Omega \subset \bR^n$ be open and bounded, $k=1,2,\cdots$.
We say $\partial \Omega$ is $C^{k,Dini}$ if for each point $x_0 \in \partial\Omega$, there exist $r>0$ independent of $x_0$ and a $C^{k, Dini}$ function (i.e., $C^k$ function whose $k$th derivatives are uniformly Dini continuous) $\gamma: \bR^{n-1} \to \bR$ such that (upon relabeling and reorienting the coordinates axes if necessary) in a new coordinate system $(x',x^n)=(x^1,\ldots,x^{n-1},x^n)$, $x_0$ becomes the origin and
\[
\Omega \cap B(0, r)=\set{ x \in B(0, r) : x^n > \gamma(x^1, \ldots, x^{n-1})},\quad  \gamma(0')=0.
\]
\end{definition}

\begin{condition}			\label{cond-n}
The coefficients $\mathbf{A}=(a^{ij})_{i,j=1}^n$ and $\vec b=(b^1,\ldots, b^n)$ are of Dini mean oscillation in $\Omega$ and  $\vec c =(c^1,\ldots, c^n)$, $d \in L^{q}(\Omega)$ with $q>n$.
\end{condition}

\begin{theorem}					\label{thm-main-d}
Let $\Omega$ have $C^{1, Dini}$ boundary and the coefficients of $L$ in \eqref{master-d} satisfy the conditions \eqref{ellipticity-d}, \eqref{bdd-d}, and Condition~\ref{cond-n}.
Let $u \in W^{1,2}_{0}(\Omega)$ be the weak solution of
\[
Lu=\dv \vec g + f\;\mbox{ in } \; \Omega,\quad u=0\;\mbox{ on }\;\partial\Omega,
\]
where $\vec g=(g^1,\ldots, g^n)$ are of Dini mean oscillation in $\Omega$ and $f \in L^q(\Omega)$ with $q>n$.
Then, we have $u\in C^1(\overline{\Omega})$.
\end{theorem}

The proof of Theorem~\ref{thm-main-d} is given Section \ref{sec2}, where an upper bound for the modulus of continuity of $Du$ can be found.

We also consider elliptic operators $\cL$ in non-divergence form
\begin{equation}							\label{master-nd}
\cL u= \sum_{i,j=1}^n a^{ij}(x) D_{ij} u+ \sum_{i=1}^n b^i(x) D_i u + c(x)u,
\end{equation}
where the coefficients $\mathbf{A}$ are assumed to be symmetric, i.e. $a^{ij}=a^{ji}$, and satisfy the uniform ellipticity and boundedness condition
\begin{equation}					\label{ellipticity-nd}
\lambda \abs{\xi}^2 \le \sum_{i,j=1}^n a^{ij}(x) \xi^i \xi^j \le \Lambda \abs{\xi}^2,\quad \forall \xi=(\xi^1,\ldots, \xi^n) \in \bR^n,\quad \forall x \in \bR^n
\end{equation}
for some constants $0< \lambda \le \Lambda$.
\begin{condition}			\label{cond-nd}
The coefficients $\mathbf{A}=(a^{ij})_{i,j=1}^n$, $\vec b=(b^1,\ldots, b^n)$, and $c$ are of Dini mean oscillation in $\Omega$.
\end{condition}

\begin{theorem}					\label{thm-main-nd}
Let $\Omega$ have $C^{2, Dini}$ boundary and the coefficients of $\cL$ in \eqref{master-nd} satisfy the condition \eqref{ellipticity-nd} and Condition~\ref{cond-nd}.
Let $u \in W^{2,2}(\Omega)\cap W^{1,2}_0(\Omega)$ be a strong solution of
\[
\cL u= g\;\mbox{ in } \; \Omega,\quad u=0\;\mbox{ on }\;\partial\Omega,
\]
where $g$ is of Dini mean oscillation in $\Omega$.
Then, we have $u\in C^2(\overline{\Omega})$.
\end{theorem}

The proof of Theorem~\ref{thm-main-nd} is also given Section \ref{sec2}.
The formal adjoint operator of the non-divergence operator $\mathcal L$ is defined by
\begin{equation*}
\mathcal L^\ast u =   D_{ij}(a^{ij}(x) u)-D_i(b^i(x) u)+c(x)u.
\end{equation*}
We also deal with the boundary value problem
\begin{equation}\label{master-adj}
\mathcal L^\ast u=\dv^2 \mathbf g + f\; \text{ in }\; \Omega,\quad u= \psi +\frac{\mathbf{g}\nu\cdot \nu}{\mathbf{A}\nu\cdot \nu} \; \text{ on }\; \partial\Omega,
\end{equation}
where $\mathbf g = (g^{kl})_{k,l=1}^n$ is a symmetric matrix and $\dv^2\mathbf g=\sum_{k,l=1}^n D_{kl}g^{kl}$.
At first, the appearance of the term $\frac{\mathbf{g}\nu\cdot \nu}{\mathbf{A}\nu\cdot \nu}$ as a part of boundary value may look strange, but it helps to make $\mathbf{g}$ to disappear from the boundary integral in the identity \eqref{eq13.53pde}, which formally defines a ``weak'' adjoint solution to \eqref{master-adj}; see
\cite[Definition~2]{EM2016} for more details.
\begin{definition}			
Let $\Omega \subset \bR^n$ be a bounded $C^{1,1}$ domain with unit exterior normal vector $\nu$.
Assume that  $\mathbf{g} \in L^p(\Omega)$, $f \in L^p(\Omega)$, and $\psi \in L^p(\partial\Omega)$, where $1<p<\infty$.
We say that $u \in L^p(\Omega)$ is an adjoint solution to \eqref{master-adj} if $u$ satisfies
\begin{equation}				\label{eq13.53pde}
\int_{\Omega} u\,\mathcal L v\,dx=\int_{\Omega} \tr(\mathbf{g}\,D^2 v)\, dx + \int_{\Omega} f v \,dx + \int_{\partial\Omega} \psi\, \mathbf{A} Dv \cdot \nu\, dS_x,
\end{equation}
for any $v \in W^{2,p'}(\Omega)\cap W^{1,p'}_0(\Omega)$, where $\frac{1}{p}+\frac{1}{p'}=1$. By a local adjoint solution of \eqref{master-adj}, we mean a function $u$ in $L^p_{loc}(\Omega)$ that verifies \eqref{eq13.53pde} when $v$ is in $W^{2,p'}_0(\Omega)$.
\end{definition}

\begin{condition}			\label{cond-adj}
The coefficients $\mathbf{A}=(a^{ij})_{i,j=1}^n$ are of Dini mean oscillation over an open set containing $\overline\Omega$ and $\vec b=(b^1,\ldots, b^n) \in L^q(\Omega)$, $c \in L^{\frac{q}{2}}(\Omega)$, for some $q>n$.
\end{condition}

\begin{theorem}					\label{thm-main-adj}
Let $\Omega$ have a $C^{1,1}$ boundary, the coefficients of $\cL$ in \eqref{master-nd} satisfy the condition \eqref{ellipticity-nd} and Condition~\ref{cond-adj}.
Let $u \in L^2(\Omega)$ be an adjoint solution of the problem
\begin{equation*}
\mathcal L^\ast u=\dv^2 \mathbf{g}+f\; \text{ in }\; \Omega,\quad u= \psi +\frac{\mathbf{g}\nu\cdot \nu}{\mathbf{A}\nu\cdot \nu} \; \text{ on }\; \partial\Omega,
\end{equation*}
where $\mathbf{g}$ is of Dini mean oscillation in $\Omega$, $f\in L^s(\Omega)$ with $s>\frac{n}{2}$, and $\psi \in C(\partial\Omega)$.
Then, $u\in C(\overline \Omega)$.
\end{theorem}

The proof of Theorem~\ref{thm-main-adj} is also given in Section \ref{sec2}.
We note that in Theorem \ref{thm-main-adj}, we assume $\mathbf{A}$ and $\mathbf{g}$ are of Dini mean oscillation, and thus $\frac{\mathbf{g}\nu\cdot \nu}{\mathbf{A}\nu\cdot \nu}$ becomes a uniformly continuous function in $\partial \Omega$.
Therefore, the boundary data in the above theorem include all continuous functions defined on $\partial\Omega$.
See \cite{S1973} for previous results on interior $C^\alpha$-regularity, $0<\alpha <1$, for solutions to \eqref{master-adj} with $\mathbf{g}=0$.

In section \ref{sec3}, we provide an improvement of the weak type-$(1,1)$ estimates given in \cite{DK16}. In particular, they are shown to hold up to the boundary, while in the non-divergence case, the weak type-$(1,1)$ estimate is shown to hold without imposing further conditions on the principal coefficients $\mathbf{A}$ other than being of Dini mean oscillation over an open set containing $\overline\Omega$.

\begin{theorem}				\label{thm-weak11-d}
Let $\Omega$ have a $C^{1,Dini}$ boundary and the coefficients $\mathbf{A}=(a^{ij})$ satisfy the conditions \eqref{ellipticity-d}, \eqref{bdd-d}, and the following:
\begin{equation}					\label{eq15.08m}
\exists c>0 \;\mbox{such that }\;\omega_{\mathbf{A}}(r) \le c ( \ln r )^{-2}, \quad\forall r\in(0,\tfrac12).
\end{equation}
Assume that $\partial \Omega$ is locally represented as a graph of $C^1$ function $\gamma$ satisfying
\[
\abs{D \gamma(x) - D\gamma(y)} \le c(\ln \abs{x-y})^{-2},\quad \abs{x-y}<\tfrac12.
\]
For $\vec f =(f^1,\ldots, f^n) \in L^2(\Omega)$, let $u \in W^{1,2}_0(\Omega)$ be a unique weak solution to
\[
\sum_{i,j=1}^n D_i(a^{ij} D_j u)= \dv \vec f\;\mbox{ in }\; \Omega,\quad
u=0 \;\mbox{ on } \; \partial \Omega.
\]
Then for any $t>0$, we have
\[
\Abs{\set{x \in \Omega : \abs{Du(x)} > t}}  \le \frac{C}{t} \int_{\Omega} \abs{\vec f},
\]
where $C=C(n, \lambda, \Lambda, \omega_{\mathbf{A}}, \Omega)$.
\end{theorem}

A similar result can be proved for the adjoint problem
\[
\sum_{i,j=1}^n D_{ij}(a^{ij} u) = \dv^2 \mathbf{g}\;\mbox{ in }\; \Omega,\quad
u=\frac{\mathbf{g}\nu\cdot\nu}{\mathbf{A}\nu\cdot\nu} \;\mbox{ on } \; \partial \Omega.
\]
The statement and its proof are similar to those of Theorem \ref{thm-weak11-d} and omitted.
\begin{theorem}				\label{thm-weak11-nd}
Let $\Omega$ have a $C^{1,1}$ boundary, the coefficients $\mathbf{A}=(a^{ij})$ have Dini mean oscillations over an open set containing $\overline\Omega$ and satisfy the condition \eqref{ellipticity-nd}.
For $f \in L^2(\Omega)$, let $u \in W^{2,2}(\Omega) \cap W^{1,2}_0(\Omega)$ be the unique solution to
\begin{equation}\label{E: nodivergenete}
\sum_{i,j=1}^n a^{ij} D_{ij} u= f\;\mbox{ in }\; \Omega,\quad
u=0 \;\mbox{ on } \; \partial \Omega.
\end{equation}
Then for any $t>0$, we have
\begin{equation}				\label{eq15.08wh}
\Abs{\set{x \in \Omega : \abs{D^2u(x)} > t}}  \le \frac{C}{t} \int_{\Omega}\, \abs{f}\,dx,
\end{equation}
where $C=C(n, \lambda, \Lambda, \omega_{\mathbf A}, \Omega)$.
\end{theorem}

We recall in Remark~\ref{R: 1} the previously known interior weak type-$(1,1)$ properties for solutions to \eqref{E: nodivergenete} and sketch out how to extend them up to the boundary, when the leading coefficients matrix $\mathbf{A}$ is only in VMO over an open set containing $\overline\Omega$.
We also explain why Theorem~\ref{thm-weak11-nd} is optimal for its comparison with counterexamples in \cite[\S 3]{Es94}.

The paper is organized as follows.
In Section~\ref{sec2} we provide some preliminary lemmas and propositions and the proofs of Theorems \ref{thm-main-d}, \ref{thm-main-nd}, and \ref{thm-main-adj}.
Section~\ref{sec3} is devoted to the proof of Theorems~\ref{thm-weak11-d} and \ref{thm-weak11-nd}.
Section~\ref{appendinx1} is an appendix where we outline for the reader's convenience a complete proof of Lemma~\ref{lem01-stein}, which is standard, and a proof of Lemma~\ref{lem8.06}, where a Harnack type inequality for nonnegative adjoint solutions are presented.
Lemma~\ref{lem01-stein} is used in the proofs of Theorems \ref{thm-weak11-d}, \ref{thm-weak11-nd} and of Lemmas \ref{lem02-weak11}, \ref{lem-weak11-nd}, and \ref{lem-weak11-adj}.

Finally, a few remarks are in order.
Theorems~\ref{thm-main-d}, \ref{thm-main-nd}, and \ref{thm-weak11-d} are easily extended to elliptic systems since their proofs do not use any scalar structure.
The same is true for Theorem~\ref{thm-main-adj} if we keep $\psi \equiv 0$ there.
In Theorem \ref{thm-main-d} (resp. Theorem \ref{thm-main-nd}), instead of assuming zero Dirichlet data, we may assume $u=\psi$ on $\partial\Omega$ with $\psi \in C^{1,Dini}(\overline{\Omega})$ (resp. $\psi \in C^{2,Dini}(\overline{\Omega})$).
Finally, the conditions on lower order terms in Theorems \ref{thm-main-d}, and \ref{thm-main-adj} can be relaxed a little.
For example, in Theorem \ref{thm-main-d} we may assume that $\vec c$, $d$, and $f$ belong to suitable Morrey-Campanato spaces instead of $L^p$ spaces.
\section{Proof of Theorems \ref{thm-main-d}, \ref{thm-main-nd}, and \ref{thm-main-adj}}			\label{sec2}
We write $x=(x^1,\ldots, x^n)=(x', x^n)$.
Hereafter, we shall denote
\[
B^{+}(0,r)=B(0,r) \cap \set{x^n >0}\quad\text{and}\quad  T(0,r)=B(0,r) \cap \set{x^n=0}.
\]
We fix a smooth (convex) domain $\cD$ satisfying $B^{+}(0,\frac12) \subset \cD \subset B^{+}(0,1)$ so that $\partial \cD$ contains a flat portion $T(0, \frac12)$.
For $\bar x \in \partial \bR^n_{+}=\set{x^n=0}$, we then set
\[
B^{+}(\bar x,r)=B^{+}(0,r)+ \bar x,\quad T(\bar x,r)=T(0,r)+\bar x,\quad\text{and}\quad \cD(\bar x,r)= r \cD+ \bar x.
\]
Throughout the rest of paper, we adopt the usual summation convention over repeated indices.
Also, for nonnegative (variable) quantities $A$ and $B$, the relation $A \lesssim B$ should be understood that there is some constant $c>0$ such that $A \le c B$.
We write $A \approx B$ if $A \lesssim B$ and $B \lesssim A$.

\subsection{Preliminary lemmas}
\begin{lemma}				\label{lem00}
Let $\Omega$ satisfy the condition \eqref{cond_a}.
If $f$ is uniformly Dini continuous and $g$ is of Dini mean oscillation in $\Omega$, then $fg$ is of Dini mean oscillation in $\Omega$.
\end{lemma}
\begin{proof}
For any $x\in \overline\Omega$ and $r>0$, we have
\begin{align*}
\fint_{\Omega(x,r)} \,\Abs{fg-\overline{fg}_{\Omega(x,r)}}
&\le \fint_{\Omega(x,r)} \,\Abs{fg-f\,\bar{g}_{\Omega(x,r)}}
+\fint_{\Omega(x,r)} \,\Abs{f\,\bar{g}_{\Omega(x,r)}-\overline{fg}_{\Omega(x,r)}}\\
&\le \sup_{\Omega(x,r)} f \cdot \omega_g(r) + \varrho_f(r) \cdot \fint_{\Omega(x,r)} \abs{g},
\end{align*}
where we used
\[
\sup_{\Omega(x,r)}\, \Abs{f\,\bar{g}_{\Omega(x,r)}-\overline{fg}_{\Omega(x,r)}}
\le \varrho_f(r) \cdot \fint_{\Omega(x,r)} \abs{g}.
\]
Therefore, we get
\begin{equation}				\label{eq11.39w}
\omega_{fg}(r) \le \norm{f}_{L^\infty(\Omega)}\, \omega_g(r)+\norm{g}_{L^\infty(\Omega)}\,\varrho_f(r)
\end{equation}
and thus $\omega_{fg}$ is a Dini function.
\end{proof}

\begin{lemma}			\label{lem02-weak11}
Let $\bar{\mathbf{A}}=(\bar a^{ij})$ be a constant matrix satisfying \eqref{ellipticity-d} and \eqref{bdd-d}.
For $\vec f \in L^2(\cD)$ let $u \in  W^{1,2}_0(\cD)$ be a unique weak solution to
\[
\sum_{i,j=1}^n D_i (\bar a^{ij} D_j u)= \dv \vec f\;\mbox{ in }\; \cD;\quad
u=0 \;\mbox{ on } \; \partial \cD.
\]
Then for any $t>0$, we have
\[
\Abs{\set{x \in \cD : \abs{Du(x)} > t}}  \le \frac{C}{t} \int_{\cD}\, \abs{\vec f},
\]
where $C=C(n, \lambda, \Lambda, \cD)$.
\end{lemma}
\begin{proof}
See proof of  \cite[Lemma 2.2]{DK16} and Lemma~\ref{lem01-stein}.
\end{proof}

\begin{lemma}			\label{lem-weak11-nd}
Let $\bar{\mathbf{A}}=(\bar a^{ij})$ be a constant symmetric matrix satisfying \eqref{ellipticity-nd}.
For $f \in L^2(\cD)$ let $u \in W^{2,2}(\cD)\cap W^{1,2}_0(\cD)$ be a unique solution to
\[
\sum_{i,j=1}^n \bar a^{ij} D_{ij} u= f\;\mbox{ in }\; \cD;\quad
u=0 \;\mbox{ on } \; \partial \cD.
\]
Then for any $t>0$, we have
\[
\Abs{\set{x \in \cD : \abs{D^2 u(x)} > t}}  \le \frac{C}{t} \int_{\cD}\, \abs{f},
\]
where $C=C(n, \lambda, \Lambda, \cD)$.
\end{lemma}
\begin{proof}
See proof of  \cite[Lemma 2.20]{DK16} and Lemma~\ref{lem01-stein}.
\end{proof}

\begin{lemma}			\label{lem-weak11-adj}
Let $\bar{\mathbf{A}}=(\bar a^{ij})$ be a constant symmetric matrix satisfying \eqref{ellipticity-nd}.
For $\mathbf g \in L^2(\cD)$ let $u \in L^{2}(\cD)$ be a unique adjoint solution to
\[
\sum_{i,j=1}^n D_{ij} (\bar a^{ij}  u)=\dv^2 \mathbf g\;\mbox{ in }\; \cD;\quad
u=\frac{\mathbf g\nu\cdot\nu}{\bar{\mathbf A}\nu\cdot\nu} \;\mbox{ on } \; \partial \cD.
\]
Then for any $t>0$, we have
\[
\Abs{\set{x \in \cD : \abs{u(x)} > t}}  \le \frac{C}{t} \int_{\cD}\, \abs{g},
\]
where $C=C(n, \lambda, \Lambda, \cD)$.
\end{lemma}
\begin{proof}
See proof of  \cite[Lemma 2.23]{DK16} and Lemma~\ref{lem01-stein}.
\end{proof}

We finish this subsection by a Lipschitz estimate for the following equation, which will be used in the proof of Theorem \ref{thm-main-adj}:
\begin{equation}
                            \label{eq1.59}
D_{ij}(\bar{a}^{ij} u)=\dv^2 \bar{\mathbf g}\;\text{ in }\; B^+(0,2),\quad
u=\frac{\bar{\mathbf g}\nu\cdot\nu}{\bar{\mathbf A}\nu\cdot\nu}\;\text{ on }\;T(0,2),
\end{equation}
where $\bar{\mathbf A}=(\bar{a}^{ij})$ and $\bar{\mathbf g}=(\bar{g}^{ij})$ are constant symmetric matrices.
\begin{lemma}
                                        \label{lem-bdry-Lip}
Let us denote $ B^+_r:=B^+(0,r)$.
Suppose that $u\in L^2(B^+_2)$ satisfies \eqref{eq1.59}.
Then for any $p\in (0,1)$ and $c\in \bR$, we have
\begin{equation}
                                    \label{eq2.02}
\norm{Du}_{L^\infty(B_{1}^+) }\le C \norm{u-c}_{L^p(B_2^+)},
\end{equation}
where $C=C(n, \lambda, \Lambda, p)$.
\end{lemma}
\begin{proof}
First we notice that $u$ is smooth in $B^+(0,2) \cup T(0,2)$ and satisfies
\[
\bar a_{ij}D_{ij} u=0\;\text{ in }\; B^+(0,2),\quad  u=\text{constant}\;\text{ on }\;T(0,2).
\]
Obviously, $u-c$ enjoys the same properties for any $c\in \bR$.
Thus, without loss of generality, we may assume that $c=0$.
By a linear transformation and a covering argument, we may further assume that $\bar a_{ij}=\delta_{ij}$, i.e. $\bar{\mathbf A}=\mathbf{I}$.
The problem is thus reduced to
\begin{equation}
                            \label{eq1.59b}
\Delta u=0\;\text{ in }\; B^+(0,2),\quad u=\text{constant}\;\text{ on }\;T(0,2).
\end{equation}
By differentiating \eqref{eq1.59b} in the tangential direction $x_k$ for $k=1,2, \ldots, n-1$, we see that $v_k=D_{k}u$ satisfies
\begin{equation*}
\Delta v_k=0\;\text{ in }\; B^+(0,2),\quad v_k=0\;\text{ on }\;T(0,2).
\end{equation*}
By classical estimates for harmonic functions, we thus have
\begin{equation}
                                    \label{eq2.09}
\norm{D_{k} u}_{L^\infty(B_{1}^+)}+\norm{DD_{k} u}_{L^\infty(B_{1}^+)} \le C \norm{D_{k} u}_{L^2(B_{3/2}^+)},\quad k=1,2, \ldots, n-1.
\end{equation}
Next, from the equation, we find that $D_{nn}u= -\sum_{k=1}^{n-1} D_{kk} u=0$ on $T(0,2)$.
Therefore, the normal derivative $v_n=D_n u$ satisfies
\begin{equation*}
\Delta v_n=0\;\text{ in }\; B^+(0,2),\quad  D_n v_n=0\;\text{ on }\; T(0,2).
\end{equation*}
Again, by classical estimates for harmonic functions, we have
\begin{equation}
                                    \label{eq2.09b}
\norm{D_{n}u}_{L^\infty(B_{1}^+)}+\norm{DD_{n}u}_{L^\infty(B_{1}^+)} \le C \norm{D_{n}u}_{L^2(B_{3/2}^+)}.
\end{equation}
Combining \eqref{eq2.09} and \eqref{eq2.09b} yields
\begin{equation}
                                    \label{eq2.13}
\norm{Du}_{L^\infty(B_{1}^+)}+ \norm{D^2 u}_{L^\infty(B_{1}^+)} \le C\norm{Du}_{L^2(B_{3/2}^+)}.
\end{equation}
Now, \eqref{eq2.02} follows from \eqref{eq2.13}, the interpolation inequality (proved by using the standard mollification technique):
\[
\norm{Du}_{L^2(B_{3/2}^+)}
\le C\norm{Du}_{L^\infty(B_{3/2}^+)}
\le \varepsilon \norm{D^2 u}_{L^\infty(B_{3/2}^+)}
+C(\varepsilon,n,p)\, \norm{u}_{L^p(B_{3/2}^+)},
\]
and a standard iteration argument.
\end{proof}

\subsection{Proof of Theorem~\ref{thm-main-d}}
First, we develop an interior $C^1$ estimate.
\begin{proposition}				\label{prop2.5p}
For any $\Omega' \subset\subset \Omega$, we have $u \in C^1(\overline \Omega{}')$.
\end{proposition}
\begin{proof}
Since $a^{ij}$ are uniformly continuous over $\overline\Omega$ and $\partial\Omega$ is $C^1$, by moving the lower-order terms to the right-hand side of the equation, we can show that $u \in W^{1, p}(\Omega)$, for any $1<p<\infty$.
Indeed, we can rewrite the equation as
\[
D_i(a^{ij} D_j u)= \dv (\vec g- u \vec b+\nabla v),\ \text{in}\ \Omega,\quad u=0,\ \text{on}\ \partial\Omega,
\]
where $v$ is the Newtonian potential of $(f- c^i D_i u - du)\chi_{\Omega}$.
By the Calder\'on-Zygmund theory (see e.g. \cite[Theorem~9.9]{GT}) and H\"older's inequality, we have
\[
\norm{D^2 v}_{L^{s}(\bR^n)} \le C \left( \norm{f}_{L^q(\Omega)} + \norm{c}_{L^q(\Omega)} \norm{\nabla u}_{L^2(\Omega)}+ \norm{d}_{L^q(\Omega)}\norm{u}_{L^2(\Omega)} \right),
\]
where $\frac{1}{s}=\frac{1}{q}+\frac{1}{2}$.
Let us set
\[
\vec g':=\vec g - u \vec b+ \nabla v
\]
and note that $\vec g' \in L^{r}(\Omega)$, where $\frac{1}{r}=\frac{1}{q}+\frac{1}{2}-\frac{1}{n}$, and
\[
\norm{\vec g'}_{L^{r}(\Omega)} \le C \left( \norm{\vec g}_{L^\infty(\Omega)} + \norm{f}_{L^q(\Omega)} \right) + C \left( \norm{\vec b}_{L^\infty(\Omega)} + \norm{c}_{L^q(\Omega)}+ \norm{d}_{L^q(\Omega)}\right) \norm{u}_{W^{1,2}(\Omega)}.
\]
On the other hand, by the energy estimates (see e.g. \cite[\S 3.4]{LU68}), we get
\[
\norm{u}_{W^{1,2}(\Omega)} \le C \norm{f}_{L^q(\Omega)}+ C\norm{\vec g}_{L^\infty(\Omega)}+ C \norm{u}_{L^1(\Omega)}.
\]
Then we apply the global $W^{1,p}(\Omega)$ theory (see, e.g. \cite[Theorem 1]{AuscherQafsaoui02})
to get
\[
\norm{u}_{W^{1,r}(\Omega)} \le C \norm{f}_{L^q(\Omega)}+ C\norm{\vec g}_{L^\infty(\Omega)} + C \norm{u}_{L^1(\Omega)},
\]
where $C$ is a constant depending only on $n$, $\lambda$, $\Lambda$, $q$, $\Omega$, $\partial \Omega$, $\omega_{\mathbf{A}}$, $\norm{\vec c}_{L^q(\Omega)}$, $\norm{d}_{L^q(\Omega)}$, and $\norm{\vec b}_{L^\infty(\Omega)}$.
Then, by standard bootstrapping argument, we have $u \in W^{1, p}(\Omega)$, for any $1<p<\infty$, and
\[
\norm{u}_{W^{1,p}(\Omega)} \le C \norm{f}_{L^q(\Omega)}+ C\norm{\vec g}_{L^\infty(\Omega)} + C \norm{u}_{L^1(\Omega)}
\]
as claimed with $C$ depending additionally on $p$.

By Morrey's inequality, we have $u \in C^{0,\mu}(\Omega)$ for any $0<\mu<1$  and
\[
\norm{u}_{C^{0,\mu}(\Omega)} \le C \norm{f}_{L^q(\Omega)}+ C\norm{\vec g}_{L^\infty(\Omega)} + C \norm{u}_{L^1(\Omega)}.
\]
Also, note that $f - c^i D_iu -du \in L^{p}(\Omega)$ for $n<p <q$ and
\[
\norm{f-c^i D_iu -du}_{L^{p}(\Omega)} \le C \left( \norm{f}_{L^q(\Omega)}+ \norm{\vec g}_{L^\infty(\Omega)} +  \norm{u}_{L^1(\Omega)} \right) \left( 1+ \norm{\vec c}_{L^{q}(\Omega)}+\norm{d}_{L^{q}(\Omega)} \right).
\]
Recall that $v$ is the Newtonian potential of $f- c^i D_i u - du$.
By \cite[Theorem~9.9]{GT} and Morrey's inequality, we find $\nabla v \in C^{0,\delta}(\bR^n)$ with $\delta=1-\frac{n}{p}$ and
\[
[\nabla v]_{C^{0,\delta}(\bR^n)}  \le C \left( \norm{f}_{L^q(\Omega)}+ \norm{\vec g}_{L^\infty(\Omega)}+ \norm{u}_{L^1(\Omega)} \right) \left( 1+ \norm{\vec c}_{L^{q}(\Omega)}+\norm{d}_{L^{q}(\Omega)} \right).
\]
Therefore, by Lemma~\ref{lem00}, we see that $\vec g'$ is of Dini mean oscillation in $\Omega$.

In summary, we see that $u$ is a weak solution of
\begin{equation*}
\dv(\mathbf{A}\nabla u)= \dv  \vec g' \;\text{ in }\;\Omega,
\end{equation*}
where $\vec g'$ is of Dini mean oscillation and $\omega_{\vec g'}$ is completely determined by the given data (namely $n$, $\lambda$, $\Lambda$, $\Omega$, $\omega_{\mathbf{A}}$, $p$, $q$, $\norm{f}_{L^q(\Omega)}$, $\norm{\vec c}_{L^q(\Omega)}$, $\norm{d}_{L^q(\Omega)}$, $\omega_b$, $\norm{\vec b}_{L^\infty(\Omega)}$, $\omega_{\vec g}$,  and $\norm{\vec g}_{L^\infty(\Omega)}$) and $\norm{u}_{L^1(\Omega)}$.
By \cite[Theorem~1.5]{DK16}, we thus find that $u \in C^1(\overline \Omega{}')$ and $\norm{u}_{C^1(\overline \Omega{}')}$ is bounded by a constant $C$ depending only on the above mentioned given data, $\norm{u}_{L^1(\Omega)}$, and $\Omega'$.
\end{proof}

Next, we turn to $C^1$ estimate near the boundary.
We shall write $B^+_r=B^+(0,r)$.
Let $\vec g'$ and $\vec g''$ be as given in the proof of Proposition~\ref{prop2.5p}.
Under a volume preserving mapping of flattening boundary
\[
y=\vec \Phi(x)=(\Phi^1(x),\ldots, \Phi^n(x)),	\quad (\det D\vec \Phi =1)
\]
let $\tilde u(y)=u(x)$,
which satisfies
\[
D_i(\tilde a^{ij} D_j \tilde u)= \dv \tilde{\vec g}'
\]
and
\begin{gather*}
\tilde a^{ij}(y)= D_l\Phi^i D_k\Phi^j a^{kl}(x), 
\quad \tilde {\vec g}'(y)= D \vec \Phi\tran \vec g'(x).
\end{gather*}
By Lemma~\ref{lem00}, we see that the coefficients $\tilde a^{ij}$ and the data $\tilde {\vec g}'$ are still of Dini mean oscillation.
Therefore, we are reduced to prove the following.
\begin{proposition}					\label{prop01}
If $u \in W^{1,2}(B^+_4)$ is a weak solution of
\[
D_i(a^{ij} D_j u)= \dv \vec g\;\mbox{ in }\;B^+_4
\]
satisfying $u=0$ on $T(0,4)$, then $u \in C^1(\overline B{}^+_1)$.
\end{proposition}

The rest of this subsection is devoted to the proof of Proposition~\ref{prop01}.
The proof of Proposition~\ref{prop01} is in the spirit of Campanato's method~\cite{Campanato} as presented in a modern textbook \cite{GM2012}.
We shall derive an a priori estimate of the modulus of continuity of $Du$ by assuming that $u$ is in $C^1(\overline B{}^+_3)$.
The general case follows from a standard approximation argument.

Fix any $p \in (0,1)$.
For $x \in B^+_4$ and  $r>0$, we define
\begin{equation}				\label{eq12.14fo}
\phi(x,r):=\inf_{\vec q \in \bR^n} \left( \fint_{B(x,r) \cap B^{+}_4} \,\abs{Du - \vec q}^p \right)^{\frac1p}
\end{equation}
and choose a vector $\vec q_{x,r}\in \bR^n$ satisfying
\begin{equation}				\label{eq0955w}
\phi(x,r) = \left( \fint_{B(x,r) \cap B^{+}_4} \,\abs{Du - \vec q_{x,r}}^p \right)^{\frac1p}.
\end{equation}
Also, for $\bar x \in T(0,4)$ and $r>0$, we introduce an auxiliary quantity
\begin{equation}				\label{eq12.14fop}
\varphi(\bar x, r):=\inf_{q\in \bR}\, \left( \fint_{B^{+}(\bar x,r) \cap B^{+}_4} \,\abs{Du - q \vec e_n}^p \right)^{\frac1p}\quad (\vec e_n=(0,\ldots,0,1) \in \bR^n)
\end{equation}
and fix a number $\bar q_{\bar x, r} \in \bR$ satisfying
\begin{equation}				\label{eq0955wp}
\varphi(\bar x, r) = \left( \fint_{B^{+}(\bar x,r) \cap B^{+}_4} \,\abs{Du - \bar q_{\bar x, r}\, \vec e_n}^p \right)^{\frac1p}.
\end{equation}
We present a series of lemmas (and their proofs) that will provide key estimates for the proof of Proposition~\ref{prop01}.
\begin{lemma}						\label{lem-01}
Let $\beta \in (0,1)$.
For any $\bar x \in T(0,3)$ and $0<\rho \le r \le \frac12$, we have
\begin{equation}					\label{eq1225f}
\varphi(\bar x, \rho) \le C\left(\frac{\rho}{r}\right)^{\beta}\,r^{-n} \norm{Du}_{L^1(B^{+}(\bar x, r))}+C \norm{Du}_{L^\infty(B^{+}(\bar x,2r))}\,\tilde\omega_{\mathbf{A}}(2\rho) + C \tilde \omega_{\vec g}(2\rho),
\end{equation}
where $C=C(n, \lambda, \Lambda, p, \beta)$ are constants and  $\tilde\omega_\bullet(t)$ is a Dini function derived from $\omega_\bullet(t)$.
\end{lemma}
\begin{proof}
Note that we have $B^{+}(\bar x, 2r) \subset B^+_4$ and
\begin{equation}				\label{eq12.14f}
\varphi(\bar x, r) \le \left(\fint_{B^{+}(\bar x,r)} \abs{Du}^p\right)^{\frac1p} \lesssim r^{-n} \norm{Du}_{L^1(B^{+}(\bar x, r))}.
\end{equation}
We decompose $u=v+w$, where $w \in W^{1,2}_0(\cD(\bar x,2r))$ is the solution of the problem
\[
\dv (\bar{\mathbf{A}} \nabla w) = -\dv ((\mathbf{A}- \bar{\mathbf{A}}) \nabla u) +\dv (\vec g-\bar{\vec g})\;\mbox{ in }\;\cD(\bar x,2r);\quad
w=0 \;\mbox{ on }\;\partial \cD(\bar x,2r).
\]
Here and below, we use the simplified notation
\[
\bar{\mathbf{A}}=\bar{\mathbf{A}}_{B^{+}(\bar x,2r)},\quad \bar{\vec g} = \bar{\vec g}_{B^{+}(\bar x,2r)}.
\]
By Lemma~\ref{lem02-weak11} with scaling, for any $t>0$, we have
\[
\abs{\set{x\in B^{+}(\bar x,r): \abs{Dw(x)} > t}}
\lesssim \frac{1}{t}\left(\norm{Du}_{L^\infty(B^{+}(\bar x,2r))} \int_{B^{+}(\bar x,2r)} \abs{\mathbf{A}-\bar{\mathbf{A}}} +  \int_{B^{+}(\bar x,2r)}  \abs{\vec g -\bar {\vec g}} \,\right),
\]
where we used $B^{+}(\bar x,r) \subset \cD(\bar x,2r) \subset B^{+}(\bar x,2r)$.
Then, we have (cf. \cite[(2.11)]{DK16})
\begin{equation}				\label{eq15.50a}
\left(\fint_{B^{+}(\bar x,r)} \abs{Dw}^p \right)^{\frac1p} \lesssim \omega_{\mathbf{A}}(2r) \,\norm{Du}_{L^\infty(B^{+}(\bar x,2r))} +  \omega_{\vec g}(2r).
\end{equation}
On the other hand, $v=u-w$ satisfies
\begin{equation}				\label{eq08.48th}
\dv(\bar{\mathbf{A}} \nabla v) =0 \;\mbox{ in }\;  B^{+}(\bar x,r);\quad v=0 \;\mbox{ on }\;  T(\bar x,r).
\end{equation}
Note that the same is satisfied by $D_j v$ for $j=1, \ldots, n-1$.
By standard boundary estimates for elliptic equations (or systems) with constant coefficients, we have
\[
\norm{D D_{j} v}_{L^\infty(B^{+}(\bar x, \frac12 r))} \lesssim r^{-1} \left(\fint_{B^{+}(\bar x,r)} \abs{D_jv}^p\,\right)^{\frac1p} \lesssim r^{-1} \left(\fint_{B^{+}(\bar x,r)} \abs{D_{x'}v}^p\,\right)^{\frac1p},
\]
where $\abs{D_{x'} v}^2 := \sum_{j=1}^{n-1} (D_j v)^2$.
Since
\[
D_{nn} v = -\frac{1}{\bar a^{nn}} 
\sum_{(i,j)\neq (n,n)}\bar a^{ij} D_{ij}v,
\]
we obtain
\begin{equation}			\label{eq9.33th}
\norm{D^2 v}_{L^\infty(B^{+}(\bar x, \frac12 r))}\lesssim r^{-1} \left(\fint_{B^{+}(\bar x, r)} \abs{D_{x'}v}^p\,\right)^{\frac1p}.
\end{equation}
Therefore, we have
\begin{equation}			\label{int-reg}
\norm{D^2 v}_{L^\infty(B^{+}(\bar x, \frac12 r))}\lesssim  r^{-1} \left(\fint_{B^{+}(\bar x,r)} \abs{Dv -q\vec e_n}^p\,\right)^{\frac1p}, \quad \forall q \in \bR.
\end{equation}
Let $0< \kappa < \frac12$ to be a number to be fixed later.
Note that we have
\[
\left(\fint_{B^{+}(\bar x,\kappa r)} \abs{D_nv - \overline{D_n v}_{B^{+}(\bar x,\kappa r)}}^p \right)^{\frac1p} \le 2\kappa r \norm{D^2 v}_{L^\infty(B(\bar x, \frac12 r))},
\]
while, for $j=1, \ldots, n-1$, we have
\[
\left(\fint_{B^{+}(\bar x,\kappa r)} \abs{D_j v}^p \right)^{\frac1p} = \left(\fint_{B^{+}(\bar x,\kappa r)} \abs{D_j v - D_jv(\bar x)}^p \right)^{\frac1p} \le 2\kappa r \norm{D^2 v}_{L^\infty(B(\bar x, \frac12 r))}.
\]
Hence, by \eqref{int-reg} we obtain
\begin{equation}				\label{eq15.50b}
\left(\fint_{B^{+}(\bar x,\kappa r)} \abs{Dv - \overline{D_n v}_{B^{+}(\bar x,\kappa r)} \, \vec e_n}^p \right)^{\frac1p}
\le C_0 \kappa \left(\fint_{B^{+}(\bar x, r)} \abs{Dv -q \vec e_n}^p \, \right)^{\frac1p},\quad \forall q \in \bR,
\end{equation}
where $C_0$ is an absolute constant determined only by $n$, $\lambda$, $\Lambda$, and $p$.
By using the decomposition $u=v+w$, we obtain from \eqref{eq15.50b} that
\begin{align*}
&\left(\fint_{B^{+}(\bar x, \kappa r)} \abs{Du - \overline{D_n v}_{B^{+}(\bar x,\kappa r)}\, \vec e_n}^p \right)^{\frac1p} \\
&\qquad\qquad \le 2^{\frac{1-p}{p}} \left(\fint_{B^{+}(\bar x, \kappa r)} \abs{Dv - \overline{D_n v}_{B^{+}(\bar x,\kappa r)}\, \vec e_n}^p \right)^{\frac1p}+ C\left(\fint_{B^{+}(\bar x, \kappa r)} \abs{Dw}^p \right)^{\frac1p}\\
&\qquad \qquad \le  4^{\frac{1-p}p}C_0\kappa \left(\fint_{B^{+}(\bar x, r)} \abs{Du - q \vec e_n}^p \right)^{\frac1p} +C (\kappa^{-\frac{n}{p}}+1) \left(\fint_{B^{+}(\bar x, r)} \abs{Dw}^p \right)^{\frac1p}.
\end{align*}
Since $q\in \bR$ is arbitrary, by using \eqref{eq15.50a}, we thus obtain
\[
\varphi(\bar x,\kappa r) \le 4^{\frac{1-p}{p}} C_0 \kappa \, \varphi(\bar x,r)+ C (\kappa^{-\frac{n}{p}}+1) \left( \omega_{\mathbf{A}}(2r) \,\norm{Du}_{L^\infty(B^{+}(\bar x,2r))} +  \omega_{\vec g}(2r) \right).
\]
For any given $\beta \in (0,1)$, let $\kappa \in (0,\frac12)$ be sufficiently small so that $4^{\frac{1-p}{p}} C_0 \le \kappa^{\beta-1}$.
Then, we obtain
\[
\varphi(\bar x,\kappa r) \le \kappa^{\beta} \varphi(\bar x,r)+ C \left( \omega_{\mathbf{A}}(2r) \,\norm{Du}_{L^\infty(B^{+}(\bar x, 2r))} +  \omega_{\vec g}(2r) \right).
\]
Note that $\kappa^\beta <1$.
By iterating, for $j=1,2,\ldots$, we get
\[
\varphi(\bar x,\kappa^j r) \le \kappa^{j \beta} \varphi(\bar x,r)
+C \left(\norm{Du}_{L^\infty(B^{+}(\bar x, 2r))} \sum_{i=1}^{j} \kappa^{(i-1)\beta} \omega_{\mathbf{A}}(2\kappa^{j-i} r) + \sum_{i=1}^{j} \kappa^{(i-1)\beta} \omega_{\vec g}(2\kappa^{j-i} r) \right).
\]
Therefore, we have
\begin{equation}				\label{eq22.25f}
\varphi(\bar x,\kappa^j r) \le \kappa^{j \beta} \varphi(\bar x,r) +C \norm{Du}_{L^\infty(B^{+}(\bar x,2r))}\,\tilde \omega_{\mathbf{A}}(2\kappa^{j} r) + C \tilde \omega_{\vec g}(2\kappa^{j} r),
\end{equation}
where we set
\begin{equation}				\label{eq14.27w}
\tilde \omega_{\bullet}(t)= \sum_{i=1}^{\infty} \kappa^{i \beta} \left(\omega_{\bullet}(\kappa^{-i}t)\, [\kappa^{-i}t \le 1] + \omega_{\bullet}(1)\, [\kappa^{-i}t >1] \right).
\end{equation}
Here, we used Iverson bracket notation; i.e., $[P] = 1$ if $P$ is true and $[P] = 0$
otherwise.
We recall that $\tilde \omega_{\bullet}(t)$ is a Dini function; see \cite[Lemma~1]{Dong2012}.

Now, for any $\rho$ satisfying $0<\rho \le r$, we take $j$ to be the integer satisfying $\kappa^{j+1} < \rho/r \le \kappa^j$.
Then, by \eqref{eq22.25f}
\begin{equation}				\label{eq0947f}
\varphi(\bar x, \rho)  \le  C \left(\frac{\rho}{r}\right)^{\beta} \varphi(\bar x, \kappa^{-j} \rho) +C \norm{Du}_{L^\infty(B^{+}(\bar x,2r))}\,\tilde \omega_{\mathbf{A}}(2\rho) + C \tilde \omega_{\vec g}(2\rho).
\end{equation}
Therefore, we get \eqref{eq1225f} from \eqref{eq0947f} and \eqref{eq12.14f}.
\end{proof}

\begin{lemma}					\label{lem-02}
Let $\beta \in (0,1)$.
For any $x \in B^+_3$ and $0<\rho\le r \le \frac14$, we have
\begin{equation}					\label{eq10.39tu}
\phi(x, \rho) \le  C\left(\frac{\rho}{r}\right)^{\beta}\, r^{-n} \norm{Du}_{L^1(B(x ,3r)\cap B^+_4)}+C\norm{Du}_{L^\infty(B(x,5r)\cap B^+_4)}\, \hat\omega_{\mathbf{A}}(\rho) + C \hat\omega_g(\rho),
\end{equation}
where $C=C(n, \lambda, \Lambda,p, \beta)$ are constants and $\hat\omega_\bullet(t)$ is a Dini function derived from $\omega_\bullet(t)$.
\end{lemma}
\begin{proof}
In this proof we shall denote
\[
\bar x= (x^1, \ldots, x^{n-1}, 0).
\]
There are three possibilities.
\begin{enumerate}[i.]
\item
$\rho \le r \le x^n$\,:\,
We utilize an interior $C^1$ estimate developed in \cite{DK16} as follows.
Since $B(x, r) \subset B^+_4$, we observe that $\phi(x,\rho)$ is identical to that introduced in the proof of \cite[Theorem 1.5]{DK16}. 
We recall that it satisfies
\[
\phi(x,\kappa r) \le 2^{\frac2p-1}C_0 \kappa \, \phi(x, r)+ C (\kappa^{-\frac{d}{p}}+1) \left( \omega_A(r) \,\norm{Du}_{L^\infty(B(x, r))} +  \omega_g(r) \right).
\]
Therefore, similar to \eqref{eq22.25f}, we get
\begin{equation}				\label{eq22.25fo}
\phi(x, \kappa^j r) \le \kappa^{j \beta} \phi(x, r) +C \norm{Du}_{L^\infty(B(x, r))}\,\tilde \omega_A(\kappa^{j} r) + C \tilde \omega_g(\kappa^{j} r),
\end{equation}
where $\tilde\omega_\bullet$ is as defined in \eqref{eq14.27w} and we take $\kappa$ in \eqref{eq22.25fo} and \eqref{eq22.25f} to be identical.
Then we get an inequality similar to \eqref{eq0947f}, namely, for $0<\rho \le r \le \frac14$
\begin{equation}				\label{eq0946fo}
\phi(x, \rho)  \le C \left(\frac{\rho}r\right)^{\beta}\, \phi(x, \kappa^{-j} \rho) +C \norm{Du}_{L^\infty(B(x,r))}\,\tilde \omega_{\mathbf{A}}(\rho) + C \tilde \omega_{\vec g}(\rho),
\end{equation}
where $j$ is the integer satisfying $\kappa^{j+1} < \rho/r \le \kappa^j$.
By \eqref{eq12.14fo}, we have
\[
\phi(x, \kappa^{-j} \rho) \le \left(\fint_{B(x, \kappa^{-j}\rho)} \abs{Du}^p\right)^{\frac1p} \le \kappa^{-n} \fint_{B(x,r)} \abs{Du}\,,
\]
and thus, we obtain
\begin{equation}				\label{eq0947fo}
\phi(x, \rho)  \le C \left(\frac{\rho}r\right)^{\beta}\, r^{-n}\norm{Du}_{L^1(B(x,r))} +C \norm{Du}_{L^\infty(B(x,r))}\,\tilde \omega_{\mathbf{A}}(\rho) + C \tilde \omega_{\vec g}(\rho).
\end{equation}

\item
$x^n \le \rho \le r$\,:\,
Since $B(x, \rho) \cap B^{+}_4 \subset B^{+}(\bar x, 2\rho) \subset B^+_4$, we have
\begin{align}					\nonumber
\phi(x, \rho) &=\left( \fint_{B(x, \rho) \cap B^{+}(0,4)} \,\abs{Du - \vec q_{x, \rho}}^p \right)^{\frac1p}  \le \left( \fint_{B(x, \rho) \cap B^{+}(0,4)} \,\abs{Du - \bar q_{\bar x, 2\rho}\, \vec e_n}^p \right)^{\frac1p} \\
							\label{eq1448fr}
&\le  2^{\frac{n}p} \left( \fint_{B^{+}(\bar x, 2\rho)} \,\abs{Du - \bar q_{\bar x, 2\rho}\,\vec e_n}^p \right)^{\frac1p} = 2^{\frac{n}{p}} \varphi(\bar x, 2\rho).
\end{align}
Therefore, by Lemma~\ref{lem-01}, and using $\abs{x-\bar x}=x^n \le r$, we obtain
\begin{align}					\nonumber
\phi(x, \rho) &\le  C \left(\frac{2\rho}{2r}\right)^\beta r^{-n}\norm{Du}_{L^1(B^{+}(\bar x, 2r))}+ C\norm{Du}_{L^\infty(B^{+}(\bar x,4r))}\, \tilde \omega_{\mathbf{A}}(4\rho) + C \tilde \omega_{\vec g}(4\rho)\\
							\label{eq0919tu}
&\le  C \left(\frac{\rho}{r}\right)^\beta r^{-n}\norm{Du}_{L^1(B(x, 3r)\cap B^+_4)}+ C\norm{Du}_{L^\infty(B(x,5r) \cap B^+_4)}\, \tilde\omega_{\mathbf{A}}(4\rho) + C \tilde\omega_{\vec g}(4\rho).
\end{align}
\item
$\rho \le x^n \le r$\,:\,
Take $R=x^n$ and let $j$ be the integer satisfying $\kappa^{j+1} < \rho/R \le \kappa^j$.\,
Since $B(x, \kappa^{-j} \rho) \subset B(x,R) \subset B^+_4$ and $B(x,R) \subset B^{+}(\bar x, 2R)$, we have
\begin{align}					\nonumber
\phi(x, \kappa^{-j} \rho) &= \left( \fint_{B(x, \kappa^{-j} \rho)} \,\abs{Du - \vec q_{x, \kappa^{-j} \rho}}^p \right)^{\frac1p} \le \left( \fint_{B(x, \kappa^{-j} \rho)} \,\abs{Du -  \bar q_{\bar x, 2R}\, \vec e_n}^p \right)^{\frac1p} \\
							\label{eq20.09fr}
&\le 2^{\frac{n-1}p} \kappa^{-\frac{n}{p}}\left( \fint_{B^{+}(\bar x, 2R)} \,\abs{Du - \bar q_{\bar x, 2R}\, \vec e_n}^p \right)^{\frac1p} = 2^{\frac{n-1}p} \kappa^{-\frac{n}{p}} \varphi(\bar x, 2R).
\end{align}
Therefore, by \eqref{eq0946fo} and Lemma~\ref{lem-01}, we get
\begin{align*}
\phi(x, \rho) &\le C\left(\frac{\rho}{R}\right)^{\beta} \varphi(\bar x, 2R) +C \norm{Du}_{L^\infty(B(x,R))}\,\tilde \omega_{\mathbf{A}}(\rho) + C \tilde \omega_{\vec g}(\rho) \\
&\le  C\left(\frac{\rho}{R}\right)^{\beta} \left\{\, \left(\frac{2R}{2r}\right)^{\beta}\,r^{-n} \norm{Du}_{L^1(B^{+}(\bar x, 2r))}+ \norm{Du}_{L^\infty(B^{+}(\bar x,4r))}\,\tilde\omega_{\mathbf{A}}(4R) +\tilde\omega_{\vec g}(4R) \right\}\\
&\qquad \qquad+C \norm{Du}_{L^\infty(B^{+}(\bar x,r))}\,\tilde \omega_{\mathbf{A}}(\rho) + C \tilde \omega_{\vec g}(\rho).
\end{align*}
Therefore, by setting
\[
\omega^{\sharp}_{\bullet}(t):= \sup_{s \in [t,1]} \,\left(\frac{t}{s}\right)^\beta\, \tilde\omega_{\bullet}(s) \quad (0<t\le 1)
\]
and using $\abs{x-\bar x}=x^n \le r$, we obtain
\begin{align}					\nonumber
\phi(x, \rho) \le  C\left(\frac{\rho}r\right)^{\beta}\,& r^{-n}\norm{Du}_{L^1(B(x ,3r)\cap B^+_4)}+C\norm{Du}_{L^\infty(B(x,5r)\cap B^+_4)}\, \omega^\sharp_{\mathbf{A}}(4\rho) + C \omega_g^\sharp(4\rho)\\
							\label{eq14.24m}
&+C\norm{Du}_{L^\infty(B(x,2r) \cap B^+_4)}\, \tilde \omega_{\mathbf{A}}(\rho) + C \tilde \omega_{\vec g}(\rho).
\end{align}

\end{enumerate}
We have covered all three possible cases and obtained bounds for $\phi(x, \rho)$, namely, \eqref{eq0947fo}, \eqref{eq14.24m}, and \eqref{eq0919tu}.
Therefore, if we set $\hat\omega_{\bullet}(t)$ as
\begin{equation*}				
\hat\omega_{\bullet}(t):= \tilde \omega_{\bullet}(t)+\tilde\omega_{\bullet}(4t)+\omega_{\bullet}^{\sharp}(4t),
\end{equation*}
then \eqref{eq10.39tu} follows.
To complete the proof, we only need to show that $\omega^{\sharp}_{\bullet}(t)$ is a Dini function.
By \cite[Lemma~1]{Dong2012}, it is enough to show
\begin{equation}				\label{eq15.39m}
\omega^{\sharp}_{\bullet}(t) \lesssim \sum_{j=0}^{\infty} \frac{1}{2^{j \beta}} \left(\tilde \omega_{\bullet}(2^j t)\, [ 2^j t \le 1] + \tilde\omega_{\bullet}(1)\, [2^j t >1] \right).
\end{equation}
To see this, we first recall $\tilde \omega_{\bullet}$ satisfies the property \eqref{eq10.31tu}.
See Remark~\ref{rmk2.33} below.
Since for any $s \in [t, 1]$, there is an integer $j$ be an integer such that $2^{j-1}t \le s < 2^j t$ and \eqref{eq15.39m} follows from the definition of $\omega^\sharp_\bullet(t)$.
\end{proof}

\begin{remark}			\label{rmk2.33}
It can be easily seen that $\tilde \omega_\bullet$ satisfies the condition \eqref{eq10.31tu}; see \cite{DK16}.
From the construction of $\hat \omega_\bullet$ in the above proof, it is routine to verify that $\hat \omega_\bullet$ satisfies the property \eqref{eq10.31tu} as well.
\end{remark}

\begin{lemma}						\label{lem-03}
We have
\begin{equation}					\label{eq10.23m}
\norm{Du}_{L^\infty(B^+_2)} \le C \norm{Du}_{L^1(B^+_4)} + C\int_0^{1} \frac{\hat \omega_{\vec g}(t)}t \,dt,
\end{equation}
where $C>0$ is a constant depending only on $n$, $\lambda$, $\Lambda$, $p$, and $\omega_{\mathbf{A}}$.
\end{lemma}
\begin{proof}
For $x \in B^+_3$ and $0<r \le \frac14$, let $\set{\vec q_{x,2^{-k}r}}_{k=0}^\infty$ be a sequence of vectors in $\bR^n$ as given in \eqref{eq0955w}.
Since we have
\[
\abs{\vec q_{x,r} - \vec q_{x, \frac12 r}}^p \le
\abs{Du(y)-\vec q_{x,r}}^p + \abs{Du(y) - \vec q_{x, \frac12 r}}^p,
\]
by taking average over $y \in B(x,\frac12 r) \cap B^{+}_4$ and then taking $p$th root, we obtain
\begin{equation}				\label{eq1426fr}
\abs{\vec q_{x,r} -\vec q_{x,\frac12 r}} \lesssim \phi(x,r) + \phi(x, \tfrac12 r).
\end{equation}
Then, by iterating, we get
\begin{equation}				\label{eq1408th}
\abs{\vec q_{x,2^{-k} r} - \vec q_{x, r}} \lesssim \sum_{j=0}^k \phi(x, 2^{-j} r).
\end{equation}
Note that \eqref{eq10.39tu} implies
\[
\lim_{k \to \infty} \phi(x,  2^{-k}r)=0,
\]
and thus, by the assumption that $u \in C^1(\bar B^{+}_3)$, we find
\[
\lim_{k \to \infty} \vec q_{x, 2^{-k}r}=Du(x).
\]
Therefore, by taking $k\to \infty$ in \eqref{eq1408th}, using \eqref{eq10.39tu} and Remark~\ref{rmk2.33}, we get
\[
\abs{Du(x)-\vec q_{x,r}} \lesssim r^{-n}\norm{Du}_{L^1(B(x ,3r)\cap B^+_4)}+
\norm{Du}_{L^\infty(B(x,5r)\cap B^+_4)} \int_0^r \frac{\hat \omega_{\mathbf{A}}(t)}t \,dt+\int_0^r \frac{\hat\omega_{\vec g}(t)}t \,dt.
\]
By averaging the obvious inequality
\[
\abs{\vec q_{x, r}}^p \le \abs{Du(y) -\vec q_{x, r}}^p + \abs{Du(y)}^p
\]
over $y \in B(x, r)\cap B^+_4$ and taking $p$th root, we get
\[
\abs{\vec q_{x, r}} \lesssim  \phi(x, r) +  \left(\fint_{B(x,r)\cap B^+_4} \abs{Du}^p\right)^{\frac1p}.
\]
Combining these together and using
\[
\phi(x,r) \lesssim r^{-n} \norm{Du}_{L^1(B(x,r)\cap B^+_4)},
\]
we obtain
\[
\abs{Du(x)}  \lesssim r^{-n} \norm{Du}_{L^1(B(x,3r)\cap B^+_4)} + \norm{Du}_{L^\infty(B(x,5r)\cap B^+_4)} \int_0^r \frac{\hat\omega_{\mathbf{A}}(t)}t \,dt+\int_0^r \frac{\hat\omega_{\vec g}(t)}t \,dt.
\]
Now, taking supremum for $x\in B(x_0, r)\cap B^+_4$, where $x_0 \in B^+_3$ and $r\le \frac14$, we have
\begin{multline*}
\norm{Du}_{L^\infty(B(x_0, r)\cap B^+_4)} \le
C r^{-n} \norm{Du}_{L^1(B(x_0, 4r)\cap B^+_4)}\\
+C \norm{Du}_{L^\infty(B(x_0, 6r)\cap B^+_4)} \int_0^r \frac{\hat\omega_{\mathbf{A}}(t)}t \,dt + C \int_0^r \frac{\hat\omega_{\vec g}(t)}t \,dt .
\end{multline*}
We fix $r_0 <\frac14$ such that for any $0<r\le r_0$,
\[
C \int_0^{r} \frac{\hat\omega_{\mathbf{A}}(t)}t \,dt \le \frac1{3^n}.
\]
Then, we have for any $x_0 \in B^+_3$ and $0<r\le r_0$ that
\[
\norm{Du}_{L^\infty(B(x_0, r)\cap B^+_4)} \le
3^{-n}\norm{Du}_{L^\infty(B(x_0,6r) \cap B^+_4)} + C r^{-n} \norm{Du}_{L^1(B(x_0,4r) \cap B^+_4)} + C \int_0^r \frac{\hat\omega_{\vec g}(t)}t \,dt.
\]
For $k=1,2,\ldots$, denote $r_k=3-2^{1-k}$.
Note that $r_{k+1}-r_k=2^{-k}$ for $k\ge 1$ and $r_1=2$.
For $x_0\in B^+_{r_k}$ and $r\le 2^{-k-3}$, we have $B(x_0, 6r) \cap B^+_4 \subset B^{+}_{r_{k+1}}$.
We take $k_0$ sufficiently large such that $2^{-k_0-3}\le r_0$.
It then follows that for any $k\ge k_0$,
\[
\norm{Du}_{L^\infty(B^+_{r_k})} \le C 2^{kn} \norm{Du}_{L^1(B^+_4)} + C \int_0^{1} \frac{\hat\omega_{\vec g}(t)}t \,dt+3^{-n} \norm{Du}_{L^\infty(B^+_{r_{k+1}})}.
\]
By multiplying the above by $3^{-kn}$ and then summing over $k \ge k_0$, we reach
\[
\sum_{k=k_0}^\infty 3^{-kn}\norm{Du}_{L^\infty(B^+_{r_k})} \le C \norm{Du}_{L^1(B^+_4)} + C \int_0^{1} \frac{\hat\omega_{\vec g}(t)}t \,dt
+\sum_{k=k_0}^\infty 3^{-(k+1)n} \norm{Du}_{L^\infty(B^+_{r_{k+1}})}.
\]
Since we assume that $u\in C^1(\overline B{}^+_3)$, the summations on both sides are convergent and we obtain \eqref{eq10.23m}.
\end{proof}

With $\tilde\omega_\bullet(t)$ and $\hat\omega_\bullet(t)$ as in Lemmas~\ref{lem-01} and \ref{lem-02}, we define
\begin{equation}				\label{eq10.46tu}
\omega^*_\bullet(t):= \hat\omega_\bullet(t)+ \int_0^t \frac{\tilde \omega_\bullet(s)}s \,ds +\tilde \omega_\bullet(4t)+ \int_0^t \frac{\tilde \omega_\bullet(4s)}s \,ds.
\end{equation}
\begin{lemma}\label{lem-04} Let $\beta \in (0,1)$.
For any $x \in B^+_3$ and $0< r \le \frac15$, we have
\begin{equation}					\label{eq11.44fr}
\abs{Du(x)-\vec q_{x,r}} \le Cr^{\beta}\,\norm{Du}_{L^1(B(x,\frac35) \cap B^+_4)} + C\norm{Du}_{L^\infty(B(x,1)\cap B^+_4)}\, \omega^*_{\mathbf{A}}(r) + C\omega_{\vec g}^*(r),
\end{equation}
where $C=C(n, \lambda, \Lambda,p, \beta)$ are constants
and $\omega^*(t)$ is defined as in \eqref{eq10.46tu}.
\end{lemma}
\begin{proof}
As in the proof of Lemma~\ref{lem-02}, we denote $\bar x= (x^1, \ldots, x^{n-1}, 0)$
and $\kappa<\frac12$ be the same constant as  \eqref{eq22.25f} and \eqref{eq22.25fo}.
Let $\set{\vec q_{x,\kappa^i r}}_{i=0}^\infty \in \bR^n$ and $\set{\bar q_{\bar x, 2\kappa^i r}}_{i=0}^\infty \in \bR$ be sequences that are chosen accordingly as in \eqref{eq0955w} and \eqref{eq0955wp}.
By using
\[
\lim_{i \to \infty} \vec q_{x,\kappa^i r}=Du(x)
\]
and a computation similar to \eqref{eq1426fr}, we get
\[
\abs{Du(x) -\vec q_{x,r}} \le \sum_{i=0}^{\infty} \,\abs{\vec q_{x, \kappa^i r}-\vec q_{x,\kappa^{i+1}r}}
\lesssim \sum_{i=0}^\infty \phi(x, \kappa^i r).
\]
Therefore, it suffices to bound $\sum_{i=0}^\infty \phi(x, \kappa^i r)$ by the right-hand side of \eqref{eq11.44fr}.

In the case when $r\le x^n$, by \eqref{eq22.25fo} and \cite[Lemma~2.7]{DK16}
we have
\[
\sum_{i=0}^\infty \phi(x, \kappa^i r) \lesssim  \phi(x,r)+ \norm{Du}_{L^\infty(B(x,\frac15) \cap B^+_4)} \int_0^{r} \frac{\tilde \omega_{\mathbf{A}}(t)}t \,dt +\int_0^{r} \frac{\tilde \omega_{\vec g}(t)}t \,dt.
\]
By Lemma~\ref{lem-02}, we get
\[
\phi(x, r)  \lesssim r^{\beta}\,\norm{Du}_{L^1(B(x,\frac35) \cap B^+_4)} + \norm{Du}_{L^\infty(B(x,1)\cap B^+_4)}\, \hat\omega_{\mathbf{A}}(r) + \hat\omega_{\vec g}(r),
\]
By combining the above two inequalities, we obtain \eqref{eq11.44fr}.

In the case when $r>x^n$, let $i_0$ be the integer such that $\kappa^{i_0+1}r  \le x^n < \kappa^{i_0} r$.
By \eqref{eq22.25fo} and \cite[Lemma~2.7]{DK16}, we have
\[
\sum_{i=i_0+1}^\infty \phi(x, \kappa^i r) \lesssim \phi(x, \kappa^{i_0+1} r) + \norm{Du}_{L^\infty(B(x,r) \cap B^+_4)} \int_0^{r} \frac{\tilde \omega_{\mathbf{A}}(t)}t \,dt +\int_0^{r} \frac{\tilde \omega_{\vec g}(t)}t \,dt.
\]
By a computation similar to \eqref{eq20.09fr}, we have
\[
\phi(x, \kappa^{i_0+1} r) \le 2^{\frac{n-1}p} \kappa^{-\frac{n}{p}} \varphi(\bar x, 2 \kappa^{i_0} r)
\]
and by \eqref{eq1448fr}, for $0 \le i \le i_0$, we have
\[
\phi(x, \kappa^i r) \le 2^{\frac{n}p} \varphi(\bar x, 2\kappa^i r).
\]
Hence, we have
\[
\sum_{i=0}^\infty \phi(x, \kappa^i r) \lesssim \sum_{i=0}^{i_0} \varphi(\bar x, 2\kappa^i r) + \norm{Du}_{L^\infty(B(x,r) \cap B^+_4)} \int_0^{r} \frac{\tilde \omega_{\mathbf{A}}(t)}t \,dt +\int_0^{r} \frac{\tilde \omega_{\vec g}(t)}t \,dt.
\]
On the other hand, by \eqref{eq22.25f} and assumption $\abs{x-\bar x}=x^n <r \le \frac15$, we have
\[
\sum_{i=0}^\infty \varphi(\bar x, 2\kappa^i r) \lesssim \varphi(\bar x, 2r)+  \norm{Du}_{L^\infty(B(x, 1) \cap B^+_4)} \int_0^{r} \frac{\tilde \omega_{\mathbf{A}}(4t)}t \,dt +\int_0^r \frac{\tilde \omega_{\vec g}(4t)}t \,dt.
\]
By Lemma~\ref{lem-01} and using $\abs{x-\bar x} \le \frac15$, we find
\[
\varphi(\bar x, 2r) \lesssim r^{\beta}\, \norm{Du}_{L^1(B(x,\frac35)\cap B^{+}_4)}+ \norm{Du}_{L^\infty(B(x,1)\cap B^+_4)}\,\tilde \omega_{\mathbf{A}}(4r) + \tilde \omega_{\vec g}(4r).
\]
Combining these together, we get \eqref{eq11.44fr} as well.
\end{proof}

Now, we are ready to show that $u \in C^1(\bar B^+_1)$.
For $x, y \in B^+_1$, we have
\[
\abs{Du(x)-Du(y)} \le \abs{Du(x)- \vec q_{x,r}}\ + \abs{\vec q_{x,r}-\vec q_{y,r}}+\abs{Du(y)-\vec q_{y,r}}.
\]
In the case when $\abs{x-y} <\frac15$, set $r=\abs{x-y}$ and apply Lemma~\ref{lem-04} to get
\[
\abs{Du(x)- \vec q_{x,r}} +\abs{Du(y)-\vec q_{y,r}}
\lesssim r^{\beta}\,\norm{Du}_{L^1(B^+_2)} + \norm{Du}_{L^\infty(B^+_2)}\, \omega^*_{\mathbf{A}}(r) + \omega_{\vec g}^*(r).
\]
Take the average over $z \in B(x,r)\cap B(y,r) \cap B^+_4$ in the inequality
\[
\abs{\vec q_{x,r}-\vec q_{y,r}}^p \le \abs{Du(z)- \vec q_{x,r}}^p +\abs{Du(z)-\vec q_{y,r}}^p
\]
and take the $p$th root and apply Lemma~\ref{lem-02} to get
 \[
\abs{\vec q_{x,r}-\vec q_{y,r}} \lesssim \phi(x,r) + \phi(y,r) \lesssim r^\beta\, \norm{Du}_{L^1(B^+_2)} + \norm{Du}_{L^\infty(B^+_2)}\, \hat\omega_{\mathbf{A}}(r) + \hat\omega_{\vec g}(r).
\]
Combining these together and using Lemma~\ref{lem-03}, we obtain (note $\hat\omega_\bullet (t) \le \omega^*_\bullet(t)$\,)
\begin{multline}					\label{eq08.33st}
\abs{Du(x)-Du(y)} \lesssim  \norm{Du}_{L^1(B^+_2)}\,\abs{x-y}^\beta \\
+\left( \norm{Du}_{L^1(B^+_4)} + \int_0^1 \frac{\hat \omega_g(t)}t\,dt \right) \omega^*_{\mathbf{A}}(\abs{x-y})+ \omega^*_{\vec g}(\abs{x-y}).
\end{multline}
In case when $\abs{x-y} \ge \frac15$, we use $\abs{Du(x) - Du(y)} \le 2\norm{Du}_{L^\infty(B^+_1)}$, apply Lemma~\ref{lem-03}, and still obtain \eqref{eq08.33st}.
This completes the proof of Proposition~\ref{prop01} and that of Theorem~\ref{thm-main-d}.
\qed

\begin{remark}
We note that the modulus of continuity estimate \eqref{eq08.33st} is sharper than the corresponding interior estimate in \cite{DK16}.
In particular, if $\mathbf A$ and $\vec g$ are H\"older continuous with exponent $\alpha \in (0,1)$, then by taking $\beta \in (\alpha, 1)$ in \eqref{eq08.33st}, one can verify that $Du$ is H\"older continuous with the same exponent $\alpha$, recovering the classical Schauder estimates.
This fact was not clear in \cite{DK16}.
\end{remark}

\subsection{Proof of Theorem~\ref{thm-main-nd}}
The idea of proof is essentially the same as that of Theorem~\ref{thm-main-d}.
We first establish interior $C^2$ estimates.

\begin{proposition}				\label{prop2.45p}
For any $p\in (1,\infty)$, we have $u\in W^{2,p}(\Omega)$.
Moreover, for any $\Omega' \subset\subset \Omega$, we have $u \in C^2(\overline \Omega{}')$.
\end{proposition}
\begin{proof}
By the $W^{2,p}$ theory, we have $u \in W^{2, p}(\Omega)$ for any $1<p<\infty$ and
\[
\norm{u}_{W^{2,p}(\Omega)} \le C \norm{g}_{L^\infty(\Omega)} + C \norm{u}_{L^1(\Omega)},
\]
where $C$ is a constant depending only on $n$, $\lambda$, $\Lambda$, $p$, $\Omega$, $\partial\Omega$, and the coefficients of $\cL$; see, for instance, \cite[Theorem 11.2.3]{Kr08}.
Therefore, by the Morrey-Sobolev embedding, $u \in C^{1,\mu}(\Omega)$ for any $0<\mu<1$  and
\[
\norm{u}_{C^{1,\mu}(\Omega)} \le C\norm{g}_{L^\infty(\Omega)}+C\norm{u}_{L^1(\Omega)}.
\]
In particular, we have
\[
\varrho_{Du}(t)+ \varrho_u(t) \le C \left( \norm{g}_{L^\infty(\Omega)} + \norm{u}_{L^1(\Omega)}\right) t^{\mu}.
\]
We rewrite the equation as
\[
a^{ij} D_{ij} u= g-b^i D_i u-cu =: g'.
\]
Then $g'$ is of Dini mean oscillation by Lemma~\ref{lem00}.
Moreover, by \eqref{eq11.39w}, we have
\[
\omega_{g'}(t) \le \omega_g(t) + C\left( \norm{g}_{L^\infty(\Omega)} + \norm{u}_{L^1(\Omega)}\right) \left\{\omega_b(t)+ \omega_c(t)+ \left(\norm{\vec b}_{L^\infty(\Omega)} + \norm{c}_{L^\infty(\Omega)}\right)t^\mu \right\}.
\]
Therefore, $\omega_{g'}$ is a Dini function that is completely determined by the given data (namely $n$, $\lambda$, $\Lambda$, $\Omega$, $\omega_{\mathbf{A}}$, $p$,  $\omega_b$, $\norm{b}_{L^\infty(\Omega)}$, $\omega_c$, $\omega_g$, and $\norm{g}_{L^\infty(\Omega)}$) and $\norm{u}_{L^1(\Omega)}$.
By \cite[Theorem~1.6]{DK16}, we thus find that $u \in C^2(\overline \Omega{}')$ and $\norm{u}_{C^2(\overline \Omega{}')}$ is bounded by a constant $C$ depending only on the above mentioned given data, $\norm{u}_{L^1(\Omega)}$, and $\Omega'$.
\end{proof}

Next, we turn to $C^2$ estimate near the boundary.
Let $g'$ be as given in the proof of Proposition~\ref{prop2.45p}.
Under a mapping of flattening boundary
\[
y=\vec \Phi(x)=(\Phi^1(x),\ldots, \Phi^n(x)),	
\]
let $\tilde u(y)=u(x)$, which satisfies
\[
\tilde a^{ij} D_{ij} \tilde u=\tilde g'-\tilde b^iD_i\tilde u=:\tilde h,
\]
where
\[
\tilde a^{ij}(y)= D_l\Phi^i D_k\Phi^j a^{kl}(x),\quad  \tilde b^i(y)= D_{kl}\Phi^i a^{kl}(x),\quad \tilde g'(y)=g'(x).
\]
By Lemmas~\ref{lem00}, we see that the coefficients $\tilde a^{ij}$ and the data $\tilde h$ are of Dini mean oscillation.
As before, we are thus reduced to prove the following.
\begin{proposition}					\label{prop01nd}
If $u \in W^{2,2}(B^+_4)$ is a strong solution of
\[
a^{ij} D_{ij} u= g\;\mbox{ in }\;B^+_4
\]
satisfying $u=0$ on $T(0,4)$, then $u \in C^2(\overline B{}^+_1)$.
\end{proposition}

The rest of this subsection is devoted to the proof of Proposition~\ref{prop01nd}.
As in the proof of Proposition~\ref{prop01}, we shall derive an a priori estimate of the modulus of continuity of $D^2u$ by assuming that $u$ is in $C^2(\overline B{}^+_3)$.

Let $\mathbb S(n)$ be the set of all symmetric $n\times n$ matrices and let
\[
\mathbb S_0(n)=\set{ \mathbf q= (q^{ij}) \in \mathbb S(n) :  q^{ij}=0\;\text{ for }\; i,j =1, 2, \ldots, n-1}.
\]
Fix any $p \in (0,1)$.
Similar to \eqref{eq12.14fo} and \eqref{eq0955w}, for $x\in B^+_4$ and $r>0$, we define
\[
\phi(x,r):=\inf_{\mathbf q \in \mathbb{S}(n)} \left( \fint_{B(x,r) \cap B^{+}_4} \,\abs{D^2u - \mathbf q}^p \right)^{\frac1p}
\]
and fix a matrix $\mathbf q_{x,r}\in \mathbb S(n)$ satisfying
\[
\phi(x,r) = \left( \fint_{B(x,r) \cap B^{+}_4} \,\abs{D^2u - \mathbf q_{x,r}}^p \right)^{\frac1p}.
\]
Also, similar to \eqref{eq12.14fop}, for $\bar x \in T(0,4)$ and $r>0$, we introduce an auxiliary quantity
\[
\varphi(\bar x, r):=\inf_{\mathbf q \in \mathbb S_0(n)}\left( \fint_{B^{+}(\bar x,r)} \abs{D^2u -\mathbf q}^p \right)^{\frac1p}.
\]

The following lemma is in parallel with Lemma~\ref{lem-01}.

\begin{lemma}						\label{lem-01nd}
Let $\beta \in (0,1)$.
For any $\bar x \in T(0,3)$ and $0<\rho \le r \le \frac12$, we have
\[
\varphi(\bar x, \rho) \le C\left(\frac{\rho}{r}\right)^{\beta}\,r^{-n} \norm{D^2u}_{L^1(B^{+}(\bar x, r))}+C \norm{D^2u}_{L^\infty(B^{+}(\bar x,2r))}\,\tilde\omega_{\mathbf{A}}(2\rho) + C \tilde\omega_g(2\rho),
\]
where $C=C(n, \lambda, \Lambda, p, \beta)$ are constants and  $\tilde\omega_\bullet(t)$ is a Dini function derived from $\omega_\bullet(t)$.
\end{lemma}
\begin{proof}
For $\bar x \in T(0,3)$ and $0<r \le \frac12$, we decompose $u=v+w$, where $w \in W^{2,2}(\cD(\bar x, 2r)) \cap W^{1,2}_0(\cD(\bar x, 2r))$ is a unique solution of the problem
\[
\bar a^{ij} D_{ij} w = -\tr ((\mathbf{A}- \bar{\mathbf{A}}) D^2 u) + g-\bar g \;\mbox{ in }\; \cD(\bar x, 2r);\quad
w=0 \;\mbox{ on }\;\partial \cD(\bar x, 2r).
\]
By Lemma~\ref{lem-weak11-nd} with scaling, we have for any $\alpha>0$,
\[
\abs{\set{x\in B^{+}(\bar x, r): \abs{D^2w(x)} > \alpha}} \lesssim \frac{1}{\alpha}\left(\,\norm{D^2 u}_{L^\infty(B^{+}(\bar x, 2r))} \int_{B^{+}(\bar x, 2r)} \abs{\mathbf{A}-\bar{\mathbf{A}}} +  \int_{B^{+}(\bar x, 2r)}  \abs{g -\bar g}\,\right).
\]
Therefore, we have
\[
\left(\fint_{B^{+}(\bar x, r)} \abs{D^2w}^p \right)^{\frac1p} \lesssim \omega_{\mathbf{A}}(2r) \,\norm{D^2 u}_{L^\infty(B^{+}(\bar x, 2r))} +  \omega_g(2r).
\]
Since $v=u-w$ satisfies
\[
\bar a^{ij} D_{ij} v= D_i (\bar a^{ij} D_j v)= \bar g\;\mbox{ in }\;B^{+}(\bar x, r);\quad v=0\;\mbox{ on }\;T(\bar x, r),
\]
we see that $D_k v$ satisfies \eqref{eq08.48th} for $k=1, \ldots, n-1$.
Therefore, by \eqref{eq9.33th}, we have
\[
\norm{D^2D_k v}_{L^\infty(B^{+}(\bar x, \frac12 r))}\lesssim r^{-1} \left(\fint_{B^{+}(\bar x, r)} \abs{D_{x'}^2 v}^p\,\right)^{\frac1p},
\]
where $\abs{D_{x'}^2 v}^2 = \sum_{k, l=1}^{n-1} (D_{kl} v)^2$.
Since $\bar a^{ij} D_{ij} (D_n v)= 0$, we find
\[
D_{nnn} v=D_{nn} D_nv = -\frac{1}{\bar a^{nn}} 
\sum_{(i,j)\neq (n,n)}\bar a^{ij} D_{ijn}v,
\]
and thus, we obtain
\[
\norm{D^3 v}_{L^\infty(B^{+}(\bar x, \frac12 r))}\lesssim r^{-1} \left(\fint_{B^{+}(\bar x, r)} \abs{D_{x'}^2 v}^p\,\right)^{\frac1p}.
\]
Therefore, similar to \eqref{int-reg}, we have
\[
\norm{D^3 v}_{L^\infty(B^{+}(\bar x, \frac12 r))}\lesssim r^{-1} \left(\fint_{B^{+}(\bar x,r)} \abs{D^2v -\mathbf q}^p\,\right)^{\frac1p},\quad \forall \mathbf q \in \mathbb S_0(n).
\]
Note that for $i, j=1, \ldots, n-1$ and $0<\kappa < \frac12$, we have
\[
\left(\fint_{B^{+}(\bar x,\kappa r)} \abs{D_{ij} v}^p \right)^{\frac1p} = \left(\fint_{B^{+}(\bar x,\kappa r)} \abs{D_{ij} v - D_{ij}v(\bar x)}^p \right)^{\frac1p} \le 2\kappa r \norm{D^3 v}_{L^\infty(B(\bar x, \frac12 r))}.
\]
Therefore, if we take $\bar {\mathbf q}_{\bar x, \kappa r} \in \mathbb S_0(n)$ whose $(i,n)$ entry is $\overline{D_{in}v}_{B^+(\bar x, \kappa r)}$ for  $i=1,\ldots, n$, then similar to \eqref{eq15.50b}, we have
\[
\left(\fint_{B^{+}(\bar x,\kappa r)} \abs{D^2v - \bar {\mathbf q}_{\bar x, \kappa r}}^p \right)^{\frac1p}
\le C_0 \kappa \left(\fint_{B^{+}(\bar x, r)} \abs{D^2v -\mathbf q}^p\, \right)^{\frac1p}, \quad \forall \mathbf q \in \mathbb S_0(n).
\]
By the same argument that led to \eqref{eq22.25f}, we find that there is $\kappa \in (0, \frac12)$ such that
\[
\varphi(\bar x, \kappa^j r) \le \kappa^{j \beta} \varphi(\bar x, r) +C \norm{D^2 u}_{L^\infty(B^{+}(\bar x, 2r))}\,\tilde \omega_{\mathbf{A}}(2\kappa^{j} r) + C \tilde \omega_g(2\kappa^{j} r),
\]
where $\tilde\omega_\bullet(t)$ is the same as in \eqref{eq14.27w}.
The rest of proof is the same as that of Lemma~\ref{lem-01}.
\end{proof}

By modifying the proof of Lemmas \ref{lem-02}, \ref{lem-03}, and \ref{lem-04} in a straightforward way, we obtain the following lemmas.

\begin{lemma}						
Let $\beta \in (0,1)$.
For any $x \in B^+_3$ and $0<\rho\le r \le \frac14$, we have
\[
\phi(x, \rho) \le  C\left(\frac{\rho}{r}\right)^{\beta}\, r^{-n} \norm{D^2u}_{L^1(B(x ,3r)\cap B^+_4)}+C\norm{D^2u}_{L^\infty(B(x,5r)\cap B^+_4)}\, \hat\omega_{\mathbf{A}}(\rho) + C \hat\omega_g(\rho),
\]
where $C=C(n, \lambda, \Lambda,p, \beta)$ are constants and $\hat\omega_\bullet(t)$ is a Dini function derived from $\omega_\bullet(t)$.
\end{lemma}

\begin{lemma}						
We have
\[
\norm{D^2 u}_{L^\infty(B^+_2)} \le C \norm{D^2 u}_{L^1(B^+_4)} + C\int_0^{1} \frac{\hat\omega_g(t)}t \,dt,
\]
where $C>0$ is a constant depending only on $n$, $\lambda$, $\Lambda$, $p$, and $\omega_{\mathbf{A}}$.
\end{lemma}

\begin{lemma}						
Let $\beta \in (0,1)$.
For any $x \in B^+_3$ and $0< r \le \frac15$, we have
\[
\abs{D^2u(x)-\mathbf q_{x,r}} \le Cr^{\beta}\,\norm{D^2u}_{L^1(B(x,\frac35) \cap B^+_4)} + C\norm{D^2u}_{L^\infty(B(x,1)\cap B^+_4)}\, \omega^*_{\mathbf{A}}(r) + C\omega_g^*(r),
\]
where $C=C(n, \lambda, \Lambda,p, \beta)$ are constants
and $\omega^*(t)$ is defined as in \eqref{eq10.46tu}.
\end{lemma}

With the above lemmas at hand, we obtain, similar to \eqref{eq08.33st}, the following estimates for $x, y \in B^+_1$:
\begin{multline*}					
\abs{D^2 u(x)-D^2 u(y)} \le  C\norm{D^2u}_{L^1(B^+_2)}\,\abs{x-y}^\beta \\
+C\left( \norm{D^2 u}_{L^1(B^+_4)} + \int_0^1 \frac{\hat\omega_g(t)}t\,dt \right)  \omega^*_{\mathbf{A}}(\abs{x-y})+ C\omega^*_g(\abs{x-y}),
\end{multline*}
where $\beta =\beta(n,\lambda, \Lambda, p)$ and $C=C(n, \lambda, \Lambda, p, \omega_{\mathbf{A}})$, and $\omega^*_\bullet(t)$ is defined as in \eqref{eq10.46tu}.
$\beta \in (0,1)$ is any given number, $C=C(n, \lambda, \Lambda, p, \omega_{\mathbf{A}}, \beta)$, and $\omega^*_\bullet(t)$ is defined as in \eqref{eq10.46tu}.
We have shown that $u \in C^2(\bar B^+_1)$ as desired.
This completes the proof of Proposition~\ref{prop01nd} and that of Theorem~\ref{thm-main-nd}.
\qed

\begin{remark}
Instead of the condition of Dini mean oscillations (in the $L^1$ sense), we may also consider coefficients and data with Dini mean oscillations in the $L^p$ sense with some $p\in (1,\infty)$, i.e.,
the function $\omega_{g, p}: \bR_+ \to \bR$ defined by
\begin{equation*}					
\omega_{g,p}(r):=\sup_{x\in \overline{\Omega}} \left(\fint_{\Omega(x,r)} \,\abs{g(y)-\bar {g}_{\Omega(x,r)}}^p\,dy \right)^{\frac 1 p}
\end{equation*}
is a Dini function.
In this case, by modifying the proofs below and using the $L^p$ estimates instead of the weak type-$(1,1)$ estimates, we can show that in the non-divergence case $D^2 u$ has Dini mean oscillations in the $L^p$ sense with the same $p$.
Similar results hold for divergence form equations, and the adjoint problem of non-divergence form equations (with the boundary data $\psi=0$).
\end{remark}

\subsection{Proof of Theorem~\ref{thm-main-adj}}    
As before, we adopt the usual summation convention over repeated indices.
We first establish the following interior $C^0$ estimates.
\begin{proposition}				\label{prop2.45p-adj}
For any $p\in (1,\infty)$, we have $u\in L^p(\Omega)$.
Moreover, for any $\Omega' \subset\subset \Omega$, we have $u \in C(\overline \Omega{}')$.
\end{proposition}
\begin{proof}
Let $w$ be a unique $W^{1,2}_0(\Omega)$ solution of
\begin{equation}			\label{eq18.16hgm}
\Delta w= D_i(b^i u)-cu +f\;\text{ in }\;\Omega,\quad w=0 \;\text{ on }\;\partial\Omega.
\end{equation}
Since $u \in L^2(\Omega)$, we have $w \in L^{r}(\Omega)$, where $\frac1r=\frac12+\frac1q-\frac1n<\frac12$.
Then, by setting $\mathbf{g}'=\mathbf{g}+w \mathbf{I} \in L^r(\Omega)$, we see that $u$ becomes an adjoint solution of
\begin{equation}			\label{eq15.55kaist}
D_{ij}(a^{ij} u)=\dv^2 \mathbf{g}'\;\text{ in }\;\Omega,\quad
 u= \psi +\frac{\mathbf{g}'\nu\cdot \nu}{\mathbf{A}\nu\cdot \nu} \; \text{ on }\; \partial\Omega.
\end{equation}
Therefore, by \cite[Lemma~2]{EM2016}, we see that $u \in L^{r}(\Omega)$.
By bootstrapping, i.e., feeding $u\in L^r(\Omega)$ back to \eqref{eq18.16hgm}, we find that $u \in L^p(\Omega)$ for any $p\in (1,\infty)$ with $\norm{u}_{L^p(\Omega)}$ controlled by the given data.
Then, by the Morrey-Sobolev embedding, we have $w \in C^{0,\mu}(\Omega)$ for some $\mu>0$ with $\norm{w}_{C^{0,\mu}(\Omega)}$ controlled by the given data, and
\[
\varrho_{w}(t) \le C\left( \norm{\mathbf{g}}_{L^\infty(\Omega)} +  \norm{f}_{L^{s}(\Omega)}+ \norm{\psi}_{L^\infty(\partial\Omega)}+ \norm{u}_{L^1(\Omega)}\right) t^\mu.
\]
Therefore, $\omega_{\mathbf g'}$ is a Dini function that is completely determined by the given data.
By \cite[Theorem~1.10]{DK16}, we thus find that $u \in C(\overline \Omega{}')$ and $\norm{u}_{C(\overline \Omega{}')}$ is bounded by a constant $C$ depending only on the given data, $\norm{u}_{L^1(\Omega)}$, and $\Omega'$.
\end{proof}

Next, we turn to continuity estimate near the boundary.
The following lemma shows that it is enough to consider the case when $\psi=0$ in \eqref{eq15.55kaist}.
\begin{lemma}				\label{lem-cont-bdry-adj}
The adjoint problem
\[
D_{ij}(a^{ij} v)=0\;\text{ in }\;\Omega,\quad v=\psi\;\text{ on }\;\partial\Omega
\]
has a unique solution $v \in C(\overline\Omega)$.
\end{lemma}
The proof of the above lemma is deferred to Section~\ref{sec3}, where we introduce normalized adjoint solutions.
Let $\mathbf{g}'$ be as in the proof of Proposition~\ref{prop2.45p-adj}.
Under a volume preserving mapping of flattening boundary
\[
y=\vec \Phi(x)=(\Phi^1(x),\ldots, \Phi^n(x))	
\]
as before, let $\tilde u(y)=u(x)$, 
which satisfies
\[
D_{ij}(\tilde a^{ij} \tilde u)=\dv^2 \tilde {\mathbf g}' +D_i(\tilde b^i \tilde u),
\]
where
\begin{gather*}
\tilde a^{ij}(y)= D_l\Phi^i D_k\Phi^j a^{kl}(x),
\quad \tilde b_i(y)=D_{kl}\Phi^i a^{kl}(x),\quad
\tilde{\mathbf{g}}'(y)= D \mathbf{\Phi}\tran \mathbf{g}' D\mathbf{\Phi} (x).
\end{gather*}
We may assume without loss of generality that $\vec\Phi$ is a $C^{1,1}$-diffeomorphism on $\bR^n$.
If we set $\tilde w$ to be a solution to
\[
\Delta \tilde w=D_i(\tilde b^i \tilde u)\;\text{ in }\;\vec\Phi(\Omega), \quad \tilde w=0\; \text{ on }\; \vec\Phi(\partial\Omega),
\]
then $\tilde u$ satisfies
\[
D_{ij}(\tilde a^{ij} \tilde u)=\dv^2 (\tilde {\mathbf{g}}'+\tilde w \mathbf{I})\;\text{ in }\;\vec\Phi(\Omega),\quad \tilde u=\frac{(\mathbf{\tilde{g}}'+\tilde w \mathbf{I})\nu\cdot\nu}{\mathbf A\nu\cdot\nu}\;\text{ on }\;\vec\Phi(\partial\Omega).
\]
By Lemma~\ref{lem00} and the proof of Proposition \ref{prop2.45p-adj}, we see that the coefficients $\tilde a^{ij}$ and the data $\tilde{\mathbf g}'+\tilde w \mathbf{I}$ are of Dini mean oscillation.
As before, we are thus reduced to prove the following.
\begin{proposition}					\label{prop01adj}
If $u \in L^{2}(B^+_4)$ is an adjoint solution satisfying
\[
D_{ij}(a^{ij} u)= \dv^2 \mathbf{g}\;\mbox{ in }\;B^+_4,\quad u=\frac{\mathbf{g}\nu\cdot \nu}{\mathbf{A}\nu\cdot\nu}\; \text{ on }\;T(0,4),
\]
then $u \in C(\overline B{}^+_1)$.
\end{proposition}

The rest of this subsection is devoted to the proof of Proposition~\ref{prop01adj}. As in the proof of Propositions~\ref{prop01} and \ref{prop01nd}, we shall derive an a priori estimate of the modulus of continuity of $u$ by assuming that $u$ is in $C(\overline B{}^+_3)$.

Fix any $p \in (0,1)$.
Similar to \eqref{eq12.14fo} and \eqref{eq0955w}, for $x\in B^+_4$ and $r>0$, we define
\[
\phi(x,r):=\inf_{q \in \bR} \left( \fint_{B(x,r) \cap B^{+}_4} \,\abs{u - q}^p \right)^{\frac1p}
\]
and fix a number $q_{x,r}\in \bR$ satisfying
\[
\phi(x,r) = \left( \fint_{B(x,r) \cap B^{+}_4} \, \abs{u - q_{x,r}}^p \right)^{\frac1p}.
\]

The following lemma is in parallel with Lemmas~\ref{lem-01} and \ref{lem-01nd}.
\begin{lemma}						
Let $\beta \in (0,1)$.
For any $\bar x \in T(0,3)$ and $0<\rho \le r \le \frac12$, we have
\[
\phi(\bar x, \rho) \le C\left(\frac{\rho}{r}\right)^{\beta}\,r^{-n} \norm{u}_{L^1(B^{+}(\bar x, r))}+C \norm{u}_{L^\infty(B^{+}(\bar x,2r))}\,\tilde\omega_{\mathbf{A}}(2\rho) + C \tilde\omega_{\mathbf{g}}(2\rho),
\]
where $C=C(n, \lambda, \Lambda, p, \beta)$ are constants and  $\tilde\omega_\bullet(t)$ is a Dini function derived from $\omega_\bullet(t)$.
\end{lemma}
\begin{proof}
For $\bar x \in T(0,3)$ and $0<r \le \frac12$, we decompose $u=v+w$, where $w \in L^2(B(\bar x, r)) $ is a unique adjoint solution of the problem
\begin{align*}
D_{ij}(\bar a^{ij}  w) &= -\dv^2 ((\mathbf{A}- \bar{\mathbf{A}}) u) + \dv^2(\mathbf{g}-\bar{\mathbf{g}}) \;\mbox{ in }\; \cD(\bar x, 2r),\\
w&=\frac{\left(\mathbf{g}-\bar{\mathbf{g}}-(\mathbf{A}- \bar{\mathbf{A}}) u\right)\nu\cdot\nu}{\bar{\mathbf{A}}\nu\cdot\nu} \;\mbox{ on }\;\partial \cD(\bar x, 2r).
\end{align*}
By Lemma~\ref{lem-weak11-adj} with scaling, we have for any $\alpha>0$,
\[
\abs{\set{x\in B^{+}(\bar x, r): \abs{w(x)} > \alpha}} \lesssim \frac{1}{\alpha}\left(\,\norm{ u}_{L^\infty(B^{+}(\bar x, 2r))} \int_{B^{+}(\bar x, 2r)} \abs{\mathbf{A}-\bar{\mathbf{A}}} +  \int_{B^{+}(\bar x, 2r)}  \abs{\mathbf{g}-\bar{\mathbf{g}}}\,\right).
\]
Therefore, we have
\[
\left(\fint_{B^{+}(\bar x, r)} \abs{w}^p \right)^{\frac1p} \lesssim \omega_{\mathbf{A}}(2r) \,\norm{u}_{L^\infty(B^{+}(\bar x, 2r))} +  \omega_{\mathbf{g}}(2r).
\]
Since $v=u-w$ satisfies
\[
D_{ij}(\bar a^{ij}  v)=\dv^2\bar{\mathbf{g}}\;\mbox{ in }\;B^{+}(\bar x, r),\quad v=\frac{\bar{\mathbf{g}}\nu\cdot\nu}{\bar{\mathbf{A}}\nu\cdot\nu}\;\mbox{ on }\;T(\bar x, r),
\]
by Lemma \ref{lem-bdry-Lip} with scaling, we have
\[
\norm{D v}_{L^\infty(B^{+}(\bar x, \frac12 r))}\lesssim r^{-1} \left(\fint_{B^{+}(\bar x,r)} \abs{v -q}^p\,\right)^{\frac1p},\quad \forall q \in \bR.
\]
Thus similar to \eqref{eq15.50b}, we have
\[
\left(\fint_{B^{+}(\bar x,\kappa r)} \abs{v - \bar {v}_{\bar x, \kappa r}}^p \right)^{\frac1p}
\le C_0 \kappa \left(\fint_{B^{+}(\bar x, r)} \abs{v -q}^p\, \right)^{\frac1p}, \quad \forall q \in \bR.
\]
By the same argument that led to \eqref{eq22.25f}, we find that there is $\kappa \in (0, \frac12)$ such that
\[
\phi(\bar x, \kappa^j r) \le\kappa^{j \beta}  \phi(\bar x, r) +C \norm{u}_{L^\infty(B^{+}(\bar x, 2r))}\,\tilde \omega_{\mathbf{A}}(2\kappa^{j} r) + C \tilde \omega_{\mathbf{g}}(2\kappa^{j} r),
\]
where $\tilde\omega_\bullet(t)$ is the same as in \eqref{eq14.27w}.
The rest of proof is the same as that of Lemma~\ref{lem-01}.
\end{proof}

By modifying the proof of Lemmas \ref{lem-02}, \ref{lem-03}, and \ref{lem-04} in a straightforward way, we obtain the following lemmas.

\begin{lemma}						
Let $\beta \in (0,1)$.
For any $x \in B^+_3$ and $0<\rho\le r \le \frac14$, we have
\[
\phi(x, \rho) \le  C\left(\frac{\rho}{r}\right)^{\beta}\, r^{-n} \norm{u}_{L^1(B(x ,3r)\cap B^+_4)}+C\norm{u}_{L^\infty(B(x,5r)\cap B^+_4)}\, \hat\omega_{\mathbf{A}}(\rho) + C \hat\omega_{\mathbf{g}}(\rho),
\]
where $C=C(n, \lambda, \Lambda,p, \beta)$ are constants and $\hat\omega_\bullet(t)$ is a Dini function derived from $\omega_\bullet(t)$.
\end{lemma}

\begin{lemma}						
We have
\[
\norm{u}_{L^\infty(B^+_2)} \le C \norm{u}_{L^1(B^+_4)} + C\int_0^{1} \frac{\hat\omega_{\mathbf{g}}(t)}t \,dt,
\]
where $C>0$ is a constant depending only on $n$, $\lambda$, $\Lambda$, $p$, and $\omega_{\mathbf{A}}$.
\end{lemma}

\begin{lemma}						
Let $\beta \in (0,1)$.
For any $x \in B^+_3$ and $0< r \le \frac15$, we have
\[
\abs{u(x)-q_{x,r}} \le Cr^{\beta}\,\norm{u}_{L^1(B(x,\frac35) \cap B^+_4)} + C\norm{u}_{L^\infty(B(x,1)\cap B^+_4)}\, \omega^*_{\mathbf{A}}(r) + C\omega_{\mathbf{g}}^*(r),
\]
where $C=C(n, \lambda, \Lambda,p, \beta)$ are constants and  $\omega^*_\bullet(t)$ is defined as in \eqref{eq10.46tu}.
\end{lemma}

With the above lemmas at hand, we obtain, similar to \eqref{eq08.33st}, the following estimates for $x, y \in B^+_1$:
\begin{multline*}					
\abs{u(x)-u(y)} \le  C\norm{u}_{L^1(B^+_2)}\,\abs{x-y}^\beta \\
+C\left( \norm{u}_{L^1(B^+_4)} + \int_0^1 \frac{\hat\omega_{\mathbf{g}}(t)}t\,dt \right)  \omega^*_{\mathbf{A}}(\abs{x-y})+ C\omega^*_{\mathbf{g}}(\abs{x-y}),
\end{multline*}
where 
$\beta \in (0,1)$ is any given number, $C=C(n, \lambda, \Lambda, p, \omega_{\mathbf{A}}, \beta)$, and $\omega^*_\bullet(t)$ is defined as in \eqref{eq10.46tu}.
We have shown that $u \in C(\bar B^+_1)$ as desired.
This completes the proof of Proposition~\ref{prop01adj} and that of Theorem~\ref{thm-main-adj}.
\qed

\section{Weak type-$(1,1)$ estimates }				\label{sec3}

\subsection{Proof of Theorem~\ref{thm-weak11-d}}
We modify the proof of \cite[Theorem~3.2]{DK16}.
Since the map $T: \vec f \mapsto Du$ is a bounded linear operator in $L^2(\Omega)$, it suffices to show that $T$ satisfies the hypothesis of Lemma~\ref{lem01-stein}.

Set $c=8$.
For fixed $x_0 \in \overline{\Omega}$ and $0<r< \frac12 \diam \Omega$, let $\vec b \in L^2(\Omega)$ be supported in $\Omega(x_0, r)$ and satisfy $\int_\Omega \vec b=0$.
Let $u \in W^{1,2}_0(\Omega)$ be the unique weak solution of
\[
\sum_{i,j=1}^n D_{i}(a^{ij} D_j u)= \dv \vec b\;\mbox{ in }\; \Omega;\quad
u=0 \;\mbox{ on } \; \partial \Omega.
\]
For any $R\ge 8r$ such that $\Omega\setminus B(x_0, R) \neq \emptyset$ and $\vec g \in C^\infty_c(\Omega(x_0, 2R)\setminus B(x_0,R))$,
let  $v \in W^{1,2}_0(\Omega)$ be a weak solution of an adjoint problem
\[
\sum_{i,j=1}^n D_i(a^{ji} D_j v)= \dv \vec g\;\mbox{ in }\; \Omega;\quad
v=0 \;\mbox{ on } \; \partial \Omega.
\]
Then, we have the identity
\begin{equation}
                            \label{eq9.24}
\int_{\Omega} Du \cdot \vec g = \int_{\Omega} \vec b \cdot Dv = \int_{\Omega(x_0 ,r)} \vec b \cdot \left(Dv-\overline{Dv}_{\Omega(x_0 ,r)} \right).	
\end{equation}
Since $\vec g=0$ in $\Omega(x_0, R)$, by flattening the boundary and using a similar argument that led to \eqref{eq08.33st}, we get
\[
\abs{Dv(x)-Dv(y)}  \lesssim\left(\left(\frac{\abs{x-y}}{R}\right)^\beta + \omega_{\mathbf{A}}^*(\abs{x-y}) \right) R^{-\frac{d}2} \norm{Dv}_{L^2(\Omega(x_0, \frac14 R))}
\]
for $x, y \in \Omega(x_0, \frac18 R)$.
Since $r\le R/8$, we thus have
\begin{align}
\norm{Dv -\overline{Dv}_{\Omega(x_0, r)}}_{L^\infty(\Omega(x_0, r))} &\lesssim R^{-\frac{d}2} \norm{Dv}_{L^2(\Omega)} \left( r^\beta R^{-\beta} + \sup_{0<t\le r\omega^*_{\mathbf{A}}(t)}\right)\label{eq15.51w}.
\end{align}
Using \eqref{eq15.08m} and \eqref{eq10.46tu}, it is routine to check that
\[
\tilde \omega_{\mathbf{A}}(r)\lesssim ( \ln r )^{-2},\quad \hat \omega_{\mathbf{A}}(r)\lesssim ( \ln r )^{-2},\quad \omega_{\mathbf{A}}^*(r)\lesssim ( \ln (4/r) )^{-1}, \quad\forall r\in(0,\tfrac12).
\]
See, for instance, \cite[Lemma~3.4]{DK16}.
By the construction of $v$, we have
\[
\norm{Dv}_{L^2(\Omega)} \lesssim \norm{\vec g}_{L^2(\Omega)}=\norm{\vec g}_{L^2(\Omega(x_0,2R)\setminus B(x_0,R))}.
\]
Therefore, we have by \eqref{eq9.24} and \eqref{eq15.51w} that
\begin{multline*}
\Abs{\int_{\Omega(x_0,2R)\setminus B(x_0,R)} Du \cdot \vec g\,}\\
\lesssim  \norm{\vec b}_{L^1(\Omega(x_0, r))}\left(r^\beta R^{-\beta} +\{\ln (4/r)\}^{-1}\right) R^{-\frac{d}2}  \norm{\vec g}_{L^2(\Omega(x_0,2R)\setminus B(x_0,R))}.
\end{multline*}
and thus, by duality and H\"older's inequality, we get
\[
\norm{Du}_{L^1(\Omega(x_0,2R)\setminus B(x_0,R))} \lesssim \left(r^\beta R^{-\beta} +\{\ln (4/r)\}^{-1}\right)\,\norm{\vec b}_{L^1(\Omega(x_0, r))}.
\]
Let $N$ be the smallest positive integer such that $\Omega \subset B(x_0, 2^{N +3}r)$. Clearly, $N \lesssim \ln(1/r)$. By taking $R=2^3r,2^4r,\ldots,2^{N+2} r$, we have
\[
\int_{\Omega \setminus B(x_0, 8r)} \abs{D u} \lesssim \sum_{k=1}^N \left(2^{-\beta k} +\{\ln (4/r)\}^{-1}\right)\ \norm{\vec b}_{L^1(B(x_0, r))} \sim \int_{B(x_0, r)} \abs{\vec b},
\]
and thus we are done.
\qed

\subsection{Proof of Theorem~\ref{thm-weak11-nd}}
For simplicity of argument, we may assume that $\Omega$ is contained in $B_5=B(0,5)$ and $\mathbf{A}$ has Dini mean oscillation on $B_{10}=B(0,10)$.

Let $W$ be be the adjoint solution to the problem
\begin{equation*}			
D_{ij}(a^{ij} W) = 0 \;\mbox{ in }\;B_{10},\quad W=1 \;\mbox{ on }\;\partial B_{10}.
\end{equation*}
It is known that $W$ is a nonnegative Muckenhoupt weight in the reverse H\"older class $\mathcal B_{\frac{n}{n-1}}(B_{10})$, with constants which depend only on $n$, $\lambda$, and $\Lambda$; i.e.,
\begin{equation}			\label{doubling}
W(B(x_0, 2r)) \lesssim W(B(x_0,r)), \quad \left( W(B(x_0,r)) := \int_{B(x_0,r)} W \,dx\right),
\end{equation}
\begin{equation}\label{E: Muckenhoupt}
\left(\fint_{B(x_0,r)} W^{\frac{n}{n-1}}\, dx\right)^{\frac{n-1}{n}} \lesssim \fint_{B(x_0,r)}W\, dx,
\end{equation}
whenever $B(x_0, 2r) \subset B_{10}$; see \cite{FS1984}. Also, $W(B_{10})\approx 1$; see \cite{Es94,FGMS1988}.
\begin{lemma}				\label{L:1}
For $x_0 \in B_1$ and $0<r \le 1$, we have
\begin{equation}		\label{eq10.53f}
\sup_{B(x_0, r)} W \le C \fint_{B(x_0, r)} W,
\end{equation}
where $C$ depends only on $n, \lambda, \Lambda$, and $\omega_{\mathbf{A}}$.
\end{lemma}
\begin{proof}
In the proof of \cite[Theorem~1.10]{DK16}, it is shown that for $x,y \in B(x_0,r)$
\[
\abs{W(x)-W(y)} \lesssim r^{-n} \norm{W}_{L^1(B(x_0, 4r))}\left( r^{-\beta} \abs{x-y}^\beta + \int_0^{\abs{x-y}} \frac{\tilde \omega_{\mathbf{A}}(t)}{t}\,dt \right),
\]
where $\tilde \omega_{\mathbf{A}}$ is as in \eqref{eq14.27w}.
Therefore, we have
\[
W(x) \lesssim r^{-n} \norm{W}_{L^1(B(x_0, 4r))}\left( r^{-\beta} \abs{x-y}^\beta + \int_0^{\abs{x-y}} \frac{\tilde \omega_{\mathbf{A}}(t)}{t}\,dt \right) + W(y).
\]
By averaging over $y \in B(x_0, r)$,
\[
W(x) \lesssim \left(1 + \int_0^{2r} \frac{\tilde \omega_{\mathbf{A}}(t)}{t}\,dt \right) \fint_{B(x_0, 4r)} W+ \fint_{B(x_0, r)} W.
\]
By using the doubling property \eqref{doubling} of $W$, we get \eqref{eq10.53f}.
\end{proof}

\begin{definition}			
We say that $\tilde v$ is a normalized adjoint solution (for the operator $a^{ij}D_{ij}$) in an open subset $\Omega'$ of $\Omega$ if $\tilde{v}$ is a continuous function defined in $\Omega'$ such that $\tilde{v} W$ is an adjoint solution in $\Omega'$, i.e., $D_{ij}(a^{ij}\tilde{v} W)=0$ in $\Omega'$.
\end{definition}
We record the following property of normalized adjoint solutions $\tilde{v}$ on $B(x_0,r)\subset \Omega$:
There are constants depending only on $\lambda, \Lambda$, and $n$ such that the following holds:
\begin{equation}\label{11.11thr}
\norm{\tilde v}_{L^\infty(B(x_0,\frac r2))} \le \frac{C}{W(B(x_0,r))} \int_{B(x_0,r)} \abs{\tilde v} W\,dx,
\end{equation}
\begin{equation*}
[\tilde v]_{C^{\alpha}(B(x_0,r))}  \le Cr^{-\alpha} \norm{\tilde v}_{L^\infty(B(x_0,2r))}.
\end{equation*}
See \cite{Ba84,Es94}.
There is also a boundary version of the above estimates.
Namely, if $\tilde v$ is a normalized adjoint solution in $\Omega(x_0, r)$ with $x_0\in \overline{\Omega}$ and $\tilde v=0$ on $B(x_0,r)\cap \partial\Omega$, then
\begin{equation}
						\label{eq14.16tu}
\norm{\tilde v}_{L^\infty(\Omega(x_0,\frac r2))} \le \frac{C}{W(B(x_0, r))} \int_{\Omega(x_0,r)} \abs{\tilde v} W
\end{equation}
and
\begin{equation}				\label{eq14.17tu}
[\tilde v]_{C^{\alpha}(B(x_0,r))} \le Cr^{-\alpha} \norm{\tilde v}_{L^\infty(\Omega(x_0,2r))}.
\end{equation}
We note that the constants $C$ and $0<\alpha\le 1$ in the above estimates depend only on $n$, $\lambda$, and $\Lambda$; see \cite{Ba84,FGMS1988}.

We are now ready to prove the theorem.
We shall make first the qualitative assumption that the coefficients $\mathbf{A}=(a^{ij})$ are smooth.
However, the constant $C$ that appears in \eqref{eq15.08wh} will not depend on the extra smoothness of the coefficients.
Let $\set{Q_l}$ be a collection of disjoint \lq\lq cubes\rq\rq\  as those used in the proof of Lemma \ref{lem01-stein} so that we have
\begin{equation}				\label{eq12.10frd}
t<\fint_{Q_l}\, \abs{f}  \le A_1 t
\end{equation}
and $\abs{f(x)} \le t$ for a.e. $x \in \Omega \setminus \textstyle \bigcup_l Q_l$.
We decompose $f=g+b$, with $b=\sum_l b_l$, such that
\[
g= \tilde m_l(f):=\frac{1}{W(Q_l)} \int_{Q_l} f W \;\text{ on }\;Q_l,
\]
$g=f$ on $\Omega\setminus \bigcup_l Q_l$, and set
\[
b_l= \chi_{Q_l}\left(f - \tilde m_l(f) \right).
\]
It is obvious that
\begin{equation}		\label{eq11.16f}
\int_{\Omega} b_l W=\int_{Q_l} b_l W=0.
\end{equation}
By \eqref{eq10.53f} we find (a ball $B(x_0, r)$ can be easily replaced by a ``cube'')
\[
\Abs{\tilde m_l(f)} \le \frac{1}{W(Q_l)}\, \norm{W}_{L^\infty(Q_l)} \int_{Q_l} \abs{f}\ \lesssim \fint_{Q_l} \abs{f}
\]
and thus, we have
\begin{equation}		\label{eq11.18f}
\int_{Q_l} \abs{b_l} \,dx \le \int_{Q_l} \abs{f} + \abs{Q_l}\,\tilde m_l(f) \lesssim \int_{Q_l} \abs{f} \lesssim  t \abs{Q_l}.
\end{equation}
Also, we find that (c.f. \eqref{a11.17f} in the Appendix)
\[
\abs{g(x)} \lesssim t,\quad\text{for a.e. }x \in \Omega.
\]
Now, we write $u=v+w$, where $v \in W^{1,2}_0(B_1)\cap W^{2,2}(B_1)$ is a unique solution to
\[
a^{ij}D_{ij} v= g\;\mbox{ in }\; \Omega,\quad
v=0 \;\mbox{ on } \; \partial \Omega.
\]
By the standard elliptic theory, we have
\[
\norm{D^2 v}_{L^2(\Omega)} \lesssim \norm{g}_{L^2(\Omega)}
\]
and thus, we have (c.f. \eqref{a10.31m} in the Appendix)
\begin{align}
					\nonumber
\Abs{\set{x \in \Omega : \abs{D^2 v(x)} > \tfrac12 t}} &\lesssim \frac{1}{t^2} \int_{\Omega} \abs{D^2 v}^2 \lesssim \frac{1}{t^2} \int_{\Omega} \abs{g}^2\ \\
					\label{eq10.31m}
&\lesssim \frac{1}{t} \int_{\Omega \setminus \bigcup_l Q_l} \abs{f} + \sum_l \,\abs{Q_l} \lesssim \frac1t \int_{\Omega} \abs{f}.
\end{align}
For each $l=1,2, \cdots$, let $w_l \in W^{1,2}_0(\Omega)\cap W^{2,2}(\Omega)$ be the unique solution to
\begin{equation}			\label{eq14.28frd}
a^{ij}D_{ij} w_l= b_l\;\mbox{ in }\; \Omega,\quad
w_l=0 \;\mbox{ on } \; \partial \Omega.
\end{equation}
Take $c=8$. We associate each $Q_l$ with a ball $B_l=B(x_l, r_l)$ as in the proof of Lemma~\ref{lem01-stein}, and denote $B_l^*=B(x_l, 8r_l)$.
Since $w =\sum_l w_l$, we have
\[
\int_{\Omega \setminus \bigcup_l B_l^*} \, \abs{D^2 w} \le \sum_l \int_{\Omega \setminus  B_l^*}\, \abs{D^2 w_l}.
\]
We claim that
\begin{equation}			\label{eq14.15f}
\int_{\Omega \setminus B_l^*} \, \abs{D^2 w_l} \lesssim \int_{Q_l} \, \abs{b_l}.
\end{equation}
Take the claim for now.
Then, by \eqref{eq11.18f} and \eqref{eq12.10frd}, we get
\[
\int_{\Omega \setminus \bigcup_l B_l^*} \,\abs{D^2 w}  \lesssim \sum_l \int_{Q_l} \, \abs{b_l} \lesssim t \sum_l \,\abs{Q_l} \lesssim \int_{\Omega}\, \abs{f},
\]
which shows that
\[
\Abs{\set{x \in \Omega: \abs{D^2 w(x)} > \tfrac12 t} \setminus \textstyle\bigcup_l Q_l^*} \lesssim \frac1t \int_{\Omega}\, \abs{f}.
\]
However,
\[
\Abs{\textstyle \bigcup_j B_l^*} \le 8^n \sum_l\, \abs{B_l} \lesssim \sum_l \,\abs{Q_l} \lesssim \frac1t \int_{\Omega}\, \abs{f}.
\]
Together then, the last two estimates imply
\[
\Abs{\set{x \in \Omega: \abs{D^2 w(x)} > \tfrac12 t}} \lesssim \frac1t \int_{\Omega} \, \abs{f},
\]
which combined with \eqref{eq10.31m} gives the theorem since $u=v+w$.

We now prove the claim \eqref{eq14.15f}.
To do this, we follow the same line of  proof of \cite[Lemma~2.20]{DK16}.
Recall that $w_l \in W^{1,2}_0(\Omega)\cap W^{2,2}(\Omega)$ satisfies \eqref{eq14.28frd} with $b_l$ supported in $Q_l\subset B_l\cap \Omega = \Omega(x_l, r_l)$.
For any $R\ge 8r_l$ such that $\Omega \setminus B(x_l, R) \neq \emptyset$ and smooth functions $\mathbf{h}=(h^{km})_{k,m=1}^n$ with a compact support in $\Omega(x_l, 2R)\setminus \overline B(x_l,R)$, let $v_l \in L^2(\Omega)$ be a unique adjoint solution of
\[
D_{ij}(a^{ij} v_l)=  \dv^2 \mathbf{h}\;\mbox{ in }\; \Omega;\quad
v_l=0 \;\mbox{ on } \; \partial \Omega,
\]
and let
\[
\tilde v_l:= v_l/W.
\]
Then, we have the identity
\begin{equation}				\label{eq22.44f}
\int_{\Omega} \tr (D^2 w_l \, \mathbf{h}) = \int_{\Omega}  v_l b_l = \int_{\Omega} \tilde v_l W b_l  = \int_{Q_l} \left(\tilde v_l-(\tilde v_l)_{x_l,r_l} \right) W b_l,
\end{equation}
where we set
\[
(\tilde v_l)_{x_l,r_l}:=\fint_{\Omega(x_l,r_l)}\, \tilde v_l
\] and used \eqref{eq11.16f}.
Since $\mathbf{h}=0$ in $\Omega(x_l,R)$, we find that $\tilde v_l$ is a normalized adjoint solution in $\Omega(x_l, R)$.
Thus, by \eqref{eq14.16tu} and \eqref{eq14.17tu}, for $x \in Q_l$, we have
\begin{align*}
\Abs{\tilde v_l(x) - (\tilde v_l)_{x_l,r_l}} & \le \fint_{\Omega(x_l,r_l)} \abs{\tilde v_l(x) - \tilde v_l(y)}\, dy \lesssim \left(\frac{r_l}{R}\right)^\alpha \norm{\tilde v_l}_{L^\infty(\Omega(x_l, \frac14 R))}\\
&\lesssim \left(\frac{r_l}{R}\right)^\alpha  \frac{1}{W(B(x_l, \frac12 R))} \int_{\Omega(x_l, R)} \,\abs{\tilde v_l} W
= \left(\frac{r_l}{R}\right)^\alpha  \frac{1}{W(B(x_l, \frac12 R))} \int_{\Omega(x_l, R)}\, \abs{v_l}.
\end{align*}
Therefore, by \eqref{eq22.44f}, \eqref{eq10.53f}, and the doubling property of $W$, we have
\begin{align*}
\Abs{\int_{\Omega(x_l,2R)\setminus B(x_l,R)}  \tr (D^2 w_l \, \mathbf{h}) }
&\lesssim \left(\frac{r_l}{R}\right)^\alpha  \frac{1}{W(B(x_l,\frac12 R))}  \norm{v_l}_{L^1(\Omega(x_l, R))} \,\norm{W}_{L^\infty(B(x_l, \frac14 R))}\, \norm{b_l}_{L^1(Q_l)}\\
&\lesssim \left(\frac{r_l}{R}\right)^\alpha  \frac{W(B(x_l, \frac14 R))}{W(B(x_l, \frac12 R))\, \abs{B(x_l, \frac14 R)}}\, \norm{v_l}_{L^1(\Omega(x_l, R))}\,\norm{b_l}_{L^1(Q_l)}\\
&\lesssim r_l^\alpha  R^{-\alpha -\frac{n}2}\,\norm{b_l}_{L^1(Q_l)} \,\norm{v_l}_{L^2(\Omega)}\\
&\lesssim r_l^\alpha  R^{-\alpha-\frac{n}2}\,\norm{b_l}_{L^1(Q_l)} \,\norm{\mathbf h}_{L^2(\Omega(x_l,2R)\setminus B(x_l,R))},
\end{align*}
where we used H\"older's inequality and the estimate
\[
\norm{v_l}_{L^2(\Omega)} \lesssim \norm{\mathbf h}_{L^2(\Omega)}
\]
in the last two inequalities.
Therefore, by duality and H\"older's inequality, we get
\[
\norm{D^2 w_l}_{L^1(\Omega(x_l,2R)\setminus B(x_l,R))} \lesssim  r_l^\alpha R^{-\alpha} \,\norm{b_l}_{L^1(Q_l)}.
\]
Now let $N$ be the smallest positive integer such that $\Omega\subset B(x_l, 2^{N} 8r_l)$.
By taking $R=2^3r_l,2^4r_l,\ldots,2^{N+2}r_l$ in the above, we have
\[
\int_{\Omega \setminus B(x_l, 8r_l)} \abs{D^2 w_l} \lesssim \sum_{k=1}^{N} 2^{-\alpha k} \norm{b_l}_{L^1(Q_l)} \simeq \int_{Q_l} \abs{b_l}
\]
as desired.

Finally, we shall show how to get rid of extra smoothness assumption on the coefficients $\mathbf{A}=(a^{ij})$.
Let $\mathbf{A}_k=(a^{ij}_k)$ and $f_k$ are smooth functions such that $a^{ij}_k$ converges uniformly to $a^{ij}$ in $\Omega$ and $f_k$ converges to $f$ in $L^2(\Omega)$ as $k \to \infty$.
We may further assume that $\mathbf{A}_k$ satisfies condition \eqref{ellipticity-nd} and that $\omega_{\mathbf{A}_k} \le \omega_{\mathbf{A}}$.
Let $u_k \in W^{1,2}_0(\Omega)\cap W^{2,2}(\Omega)$ satisfy
\[
a^{ij}_k D_{ij} u_k= f_k\;\mbox{ in }\; \Omega,\quad
u_k=0 \;\mbox{ on } \; \partial \Omega.
\]
Then, by \eqref{eq15.08wh}, we have an estimate
\begin{equation}				\label{eq15.59wh}
\Abs{\set{x \in \Omega : \abs{D^2u_k(x)} > t}}  \le \frac{C}{t} \int_{\Omega}\, \abs{f_k}\,dx,
\end{equation}
which is uniform for all $k=1,2,\cdots$.
Note that
\[
a^{ij}_k D_{ij} (u-u_k) = f-f_k + (a^{ij}_k-a^{ij}) D_{ij} u
\]
so that, by the $W^{2,2}$ theory, we have
\[
\norm{D^2 u - D^2 u_k}_{L^2(\Omega)} \le C \left( \norm{f-f_k}_{L^2(\Omega)}+ \norm{(a^{ij}_k-a^{ij})D_{ij}u}_{L^2(\Omega)} \right).
\]
In particular, we find that $D^2 u_k \to D^2 u$ in $L^2(\Omega)$.
Therefore, by taking limit $k\to \infty$ in \eqref{eq15.59wh}, we get \eqref{eq15.08wh}.
\qed
\begin{remark}\label{R: 1} It is shown in \cite{Es94} that when the coefficients of the elliptic operator $\mathcal L$ in non-divergence form are VMO functions in $B_{10}$, i.e.,
\begin{equation}			\label{E:20}
\lim_{r\to 0 +} \sup_{x\in B_{10}}\fint_{Q(x,r)} \Abs{\mathbf{A}-\bar{\mathbf{A}}_{Q(x,r)}} =0,
\end{equation}
where $Q(x,r)$ denotes a cube in $\bR^n$ centered at $x$ with edges of length $r$ and sides parallel to the coordinate axis, the Muckenhoupt property \eqref{E: Muckenhoupt} of the weight $W$ improves because then, $\log W$ is in $\text{VMO}(B_5)$ (see \cite[Theorem 1.2]{Es94}), i.e.,
\begin{equation}\label{E: estaenvmo}
\lim_{r\to 0+} \sup_{x\in B_5}\fint_{Q(x,r)} \Abs{\log{W}-\overline{\log{W}}_{Q(x,r)}} =0,
\end{equation}
while the weight $W$ is shown to be unbounded or to vanish inside $B_{10}$ for some of these operators; see \cite[\S 3]{Es94}. In fact, when $\mathbf{A}$ has Dini mean oscillation in $B_{10}$, the ideas behind \cite[Theorem 1.10]{DK16} and \cite[Lemma 3]{EM2016} imply that nonnegative adjoint solutions $W$ verify the following Harnack inequality: there is $C=C(\lambda, \Lambda, n, \omega_\mathbf{A})$ such that
\[
\sup_{B(x,r)}W\le C\inf_{B(x,r)}W,
\]
when $B(x,2r)\subset B_8$.
See Lemma \ref{lem8.06} in the Appendix.
This local Harnack inequality fails when $\mathbf{A}$ is continuous on $B_{10}$; see \cite[\S 3]{Es94}.

It is shown in  \cite[Theorem 1.3] {Es94} that under the same hypothesis, the solution $u$ to \eqref{E: nodivergenete} verifies an interior weak type-$(1,1)$ property with respect to the weight $W\,dx$, when $f\in C_c^\infty(\Omega)$.
Also, \cite[\S 3]{Es94} provides counterexamples of operators $\mathcal L$ in non-divergence form with continuous coefficients for which the weak type-$(1,1)$ \eqref{eq15.08wh} fails in the interior of $\Omega$.
Notice that the coefficient matrices there just fail to be Dini continuous or have Dini mean oscillations.

Finally, the methods in this paper, \eqref{doubling}, \eqref{11.11thr} -- \eqref{eq14.17tu} and \eqref{E: estaenvmo} are easily seen to help to extend up to the boundary the interior weak type-$(1,1)$ property with respect to the weight $W\,dx$ in \cite[Theorem 1.3] {Es94}.
In particular, under the weaker VMO condition \eqref{E:20} one can show that there is $C=C(n, \lambda, \Lambda, \Omega,\mathbf{A})$ such that for any $t>0$,
\begin{equation*}
W\left(\{x \in \Omega : \abs{D^2u(x)} > t\}\right)  \le \frac{C}{t} \int_{\Omega}\, \abs{f} W\,dx,
\end{equation*}
when $u$ is the solution to \eqref{E: nodivergenete}, $\Omega\subset B_5$, and $f\in C^\infty(\overline{\Omega})$.
Observe there are operators with uniformly continuous coefficients such that the adjoint solution $W$ is unbounded above or below or it is not a local $A_1$ Muckenhoupt weight (See \cite{Ba84b} and \cite[\S 3]{Es94}).
\end{remark}

\subsection{Proof of Lemma~\ref{lem-cont-bdry-adj}}

As in the proof of Theorem~\ref{thm-weak11-nd}, let us assume that $\Omega$ is contained in $B_5=B(0,5)$ and $\mathbf{A}$ has Dini mean oscillation on $B_{10}$.
Let $W$ be as given in the proof of Theorem~\ref{thm-weak11-nd}.
By \cite[Theorem~1.10]{DK16}, we find that $W$ is uniformly continuous in $B_5$ with its modulus of continuity determined by $n$, $\lambda$, $\Lambda$, and $\omega_{\mathbf A}$.
Also, Lemma~\ref{lem8.06} in the Appendix implies that $W$ is bounded from above and below in $B_5$ with its lower and upper bounds depending only on $n$, $\lambda$, $\Lambda$, and $\omega_{\mathbf A}$.

Therefore, by \cite[Theorem 2.8]{Es00}, there is a unique normalized adjoint solution $\tilde v$ that satisfies
\[
D_{ij}(a^{ij}\tilde v W)=0\; \text{ in }\;\Omega,\quad \tilde v = \frac{\psi}{W}\; \text{ on }\;\partial\Omega.
\]
Moreover, $\tilde v \in C(\overline\Omega)$ with a modulus of continuity controlled by $n$, $\lambda$, $\Lambda$, the Lipschitz character of $\Omega$, and the modulus of continuity of $\frac{\psi}{W}$.
The latter in turn is controlled by $n$, $\lambda$, $\Lambda$, $\omega_{\mathbf A}$, and $\psi$.
It is clear that $v=\tilde v W$ satisfies all the desired properties. \qed

\section{Appendix}\label{appendinx1}
The following lemma is a slight generalization of \cite[Lemma~2.1]{DK16}.
For the completeness, we present a proof here.
\begin{lemma}		\label{lem01-stein}
Let $\Omega\subset \bR^n$ be a bounded domain satisfying the condition \eqref{cond_a} and let $T$ be a bounded linear operator from $L^2(\Omega)$ to $L^2(\Omega)$.
Let $\mu \in (0,1)$ be a constant.
Suppose that for any $x_0 \in \Omega$ and $0<r< \mu \diam \Omega$, we have
\begin{equation}					\label{eq14.58w}
\int_{\Omega \setminus B(x_0, cr)} \abs{Tb} \le C \int_{B(x_0, r)\cap \Omega} \abs{b}
\end{equation}
whenever $b \in L^2(\Omega)$ is supported in $B(x_0, r)\cap \Omega$, $\int_{\Omega} b =0$, and $c>1$ and $C>0$ are constants.
Then for $f \in L^2(\Omega)$ and any $t>0$, we have
\begin{equation}					\label{eq13.38w}
\Abs{\set{x \in \Omega : \abs{Tf(x)} > t}}  \le \frac{C'}{t} \int_{\Omega} \abs{f},
\end{equation}
where $C'=C'(n, c , C, \mu, \Omega, A_0)$ is a constant.
\end{lemma}

\begin{proof}
To begin with, we note that $\Omega$ equipped with the standard Euclidean metric and the Lebesgue measure (restricted to $\Omega$) is a space of homogeneous type.
By \cite[Theorem~11]{Ch90}, there exists a collection of open subsets (called ``cubes'')
\[
\set{Q^k_\alpha \subset \Omega :  k \in \bZ, \; \alpha \in I_k},
\]
with $I_k$ at most countable set and constants $\delta \in (0,1)$, $a_0>0$ and $C_1<\infty$ such that
\begin{enumerate}[i)]
\item
$\Abs{\Omega \setminus \bigcup_\alpha Q^k_\alpha} =0\quad \forall k$.
\item
If $\ell \ge k$ then either $Q^\ell_\beta \subset Q^k_\alpha$ or $Q^\ell_\beta \cap Q^k_\alpha=\emptyset$.
\item
For each $(k,\alpha)$ and each $\ell < k$ there is a unique $\beta$ such that $Q^k_\alpha \subset Q^\ell_\beta$.
\item
$\diam Q^k_\alpha \le C_1 \delta^k$.
\item
Each $Q^k_\alpha$ contains some \lq\lq ball\rq\rq\ $B(z^k_\alpha, a_0 \delta^k) \cap \Omega$.
\end{enumerate}
From the above, we can infer the following.
\begin{enumerate}[(a)]
\item
There is constant $A_1\ge 1$ such that if $Q^{k-1}_\beta$ is the parent of $Q^k_\alpha$ (resp. if $B^k_\alpha$ is the Euclidean ball in $\bR^n$ centered at $z^k_\alpha$ with radius $r=\diam Q^k_\alpha$\,), then we have
\begin{equation}				\label{p1}
\abs{Q^{k-1}_\beta} \le A_1 \abs{Q^k_\alpha}\quad (\text{resp. }\;\abs{B^{k}_\alpha} \le A_1 \abs{Q^k_\alpha}\,).
\end{equation}
\item
The Lebesgue differentiation theorem is available for the chain of cubes shrinking to a point because the maximal function defined as
\begin{equation*}
M(f)(x)=
\begin{cases}
\sup_{x\in Q^k_\alpha}\fint_{Q^k_\alpha}\abs f\, dx,\ &\text{when}\ x\in\bigcap_k\bigcup_{\alpha\in I_k}Q^k_\alpha,\\
0,\ &\text{otherwise},
\end{cases}
\end{equation*}
is of weak type-$(1,1)$ over $\Omega$.
\end{enumerate}
By i) -- v) above and \eqref{cond_a}, choose $k_0\in \bZ$ with $\theta:=\inf_{\alpha\in I_{k_0}}|Q^{k_0}_\alpha|>0$. To get \eqref{eq13.38w}  when
\[\frac 1t\int_{\Omega} \abs{f}\ dx>\theta,\]
it suffices to choose $C'\ge \theta^{-1}\abs{\Omega}$.
Otherwise,
\[
\fint_{Q^{k_0}_\alpha}\abs{f}\ dx\le t,\ \text{for all}\ \alpha\in I_{k_0}.
\]
Let then $\set{Q_l}$ denote the set of cubes chosen as follows.
For $k=k_0+1$ and $\alpha \in I_k$, the cube $Q=Q^k_\alpha$ satisfies either
$\fint_{Q}\, \abs{f} \le t$ or $\fint_{Q}\, \abs{f} > t$.
In the second case, we select $Q=Q^k_\alpha$ as one of the cubes in $\set{Q_l}$.
Note that in this case, we have by \eqref{p1}
\[
t<\fint_{Q} \abs{f}\,dx  \le A_1 t.
\]
In the first case, we subdivide $Q=Q^k_\alpha$ further into subcubes $Q'=Q^{k+1}_\beta$, and repeat the process until (if ever) we are forced into the second case.
By observation (b), we find that $\abs{f(x)} \le t$ for a.e. $x \in \Omega \setminus \bigcup_l Q_l$.

We decompose $f=g+b$, with $b=\sum_l b_l$, such that
\[
g=  m_l(f):=\fint_{Q_l} f \,dx\;\text{ on }\;Q_l,
\]
$g=f$ on $\Omega \setminus \bigcup_l Q_l$, and set
\[
b_l= \chi_{Q_l}\left(f - m_l(f) \right).
\]
It is obvious that $\int_{Q_l} b_l \,dx =0$ and we have
\begin{equation}		\label{a11.18f}
\int_{Q_l} \abs{b_l} \,dx \le \int_{Q_l} \abs{f}\,dx + \abs{Q_l} m_l(f) \le 2 \int_{Q_l} \abs{f} \,dx \le 2A_1 t \abs{Q_l}.
\end{equation}
Also, we see that
\begin{equation}		\label{a11.17f}
\abs{g(x)} \le A_1 t \quad\text{for a.e. }x \in \Omega.
\end{equation}
Indeed, for a.e. $x \in \Omega \setminus  \bigcup_l Q_l$, we have $\abs{g(x)} =\abs{f(x)} \le t$ and $\abs{g(x)} \le A_1 t$ on $Q_l$.
By Chebyshev's inequality and the $L^2$ boundedness of $T$, we have
\begin{align}
					\nonumber
\Abs{\set{x \in \Omega : \abs{Tg(x)} > \tfrac12 t}} &\lesssim \frac{1}{t^2} \int_{\Omega}\, \abs{Tg}^2\,dx \lesssim \frac{1}{t^2} \int_{\Omega}\, \abs{g}^2\,dx \\
					\label{a10.31m}
&\lesssim \frac{1}{t} \int_{\Omega \setminus \bigcup_l Q_l} \abs{f}\,dx + \sum_l \,\abs{Q_l} \lesssim \frac1t \int_{\Omega} \abs{f}\,dx,
\end{align}
where we used \eqref{a11.17f} and the property that
\begin{equation}				\label{a14.46m}
\sum_l \,\abs{Q_l} \le \frac{1}{t} \int_{\Omega} \abs{f}\,dx.
\end{equation}
We associate each $Q_l=Q^k_\alpha$ with a Euclidean ball $B_l=B(x_l, r_l)$, where $x_l=z^k_\alpha\in \Omega$ and $r_l=\diam Q^k_\alpha$.
Let us denote $B_l^*=B(x_l, cr_l)$.
Since $Tb =\sum_l Tb_l$, we have
\[
\int_{\Omega \setminus \bigcup_l B_l^*} \abs{Tb}\,dx \le \sum_l \int_{\Omega \setminus  B_l^*}  \abs{T b_l}\,dx.
\]
By the hypothesis \eqref{eq14.58w} together with \eqref{a11.18f} and \eqref{a14.46m}, we get
\[
\int_{\Omega \setminus \bigcup_{l} B_l^*} \abs{Tb}\,dx \le C \sum_l \int_{Q_l}  \abs{b_l}\,dx \lesssim t \sum_l \abs{Q_l} \lesssim \int_{\Omega} \abs{f}\,dx,
\]
which via Chebyshev's inequality shows that
\[
\Abs{\set{x \in \Omega: \abs{Tb(x)} > \tfrac12 t} \setminus \textstyle\bigcup_l B_l^*\,} \lesssim \frac1t \int_{\Omega} \abs{f}\,dx.
\]
Also, by \eqref{p1}, we have
\[
\Abs{\textstyle \bigcup_l B_l^*} \le c^n \sum_l\, \abs{B_l}\  \lesssim \sum_l \,\abs{Q_l} \lesssim \frac1t \int_{\Omega} \abs{f}\,dx.
\]
Together then, the last two estimates imply
\[
\Abs{\set{x \in \Omega: \abs{Tb(x)} > \tfrac12 t}} \lesssim \frac1t \int_\Omega \abs{f}\,dx,
\]
which combined with \eqref{a10.31m} gives \eqref{eq13.38w} since $Tf=Tg+Tb$.
\end{proof}

Finally we prove the following Harnack type inequality for nonnegative adjoint solutions.

\begin{lemma}				\label{lem8.06}
Assume the coefficients $\mathbf{A}=(a^{ij})$ are of Dini mean oscillation and satisfy the condition \eqref{ellipticity-nd}.
Let $w\in L^2(B_4)$ be a nonnegative solution to $D_{ij}(a^{ij} w)=0$ in $B_4=B(0,4)$ and $\norm{w}_{L^1(B_3)}=1$.
Then we have
\[
c\le \inf_{B_1}\,w,\quad \sup_{B_1}\,w \le C,
\]
where $c$ and $C$ are positive constants depending only on $n$, $\lambda$, $\Lambda$, and $\omega_{\mathbf{A}}$.
\end{lemma}
\begin{proof}
The upper bound follows with the same type of reasoning as in the proof of Lemma \ref{L:1}, because from \cite[(2.25)]{DK16}, we have
\begin{equation}
                                            \label{eq2.45}
\abs{w(x)-w(y_0)}  \le C \left(\left(\frac{\abs{x-y_0}}{R}\right)^\beta + \int_0^{\abs{x-y_0}} \frac{\tilde\omega_{\mathbf{A}}(t)}t \,dt \right) R^{-d} \norm{w}_{L^1(B(y_0,R))}
\end{equation}
for $x \in B(y_0, \frac12 R)$, $y_0\in B_1$ and $R\in (0,1]$.
Here $\beta>0$ is an absolute constant and $\tilde \omega_{\mathbf{A}}$ is defined as in \eqref{eq14.27w}.

We prove the lower bound by contradiction.
Suppose the claim is not true.
Then we can find a sequence of coefficients $\set{\mathbf{A}_k}$ satisfying
\[
\sup_k \omega_{\mathbf{A}_k}(t)  \le \omega(t)
\]
for some Dini function $\omega$ and a sequence of nonnegative solutions $\set{w_k}$ to
$$
D_{ij}(a^{ij}_k w_k)=0\quad \text{in}\,\, B_4
$$
such that
\[
\norm{w_k}_{L^1(B_3)}=1\quad\text{and}\quad w_k(x_k) \le  1/k
\]
for some $x_k\in B_1$.
After passing to a subsequence, we may assume that $x_k\to y_0 \in \bar B_1$.
By \cite[Theorem 1.10]{DK16}, $\set{w_k}$ is uniformly bounded and equicontinuous in $\bar B_2$.
Of course, $\set{\mathbf{A}_k}$ is also uniformly bounded and equicontinuous in $\bar B_2$.
Therefore, by the Arzel\`a--Ascoli theorem, they have subsequences, still denoted by $\set{w_k}$ and $\set{\mathbf{A}_k}$, which converge to $w$ and $\mathbf{A}$ uniformly in $\bar B_2$, with the same moduli of continuity.
It is easily seen that $w$ is a nonnegative solution of
\[
D_{ij}(a^{ij}w)=0 \quad \text{in}\quad B_2
\]
and $w(y_0)=0$.
By the doubling property of $w$ (see \cite{FS1984}), $\norm{w_k}_{L^1(B_2)}$ is bounded from below and above uniformly. It then follow from the uniform convergence that  $\norm{w}_{L^1(B_2)}$ is also bounded from below and above.

Let $\kappa\in (0,1/2)$ be a small constant to be specified later.
From \eqref{eq2.45}, for any $R\in (0,1]$, we have
\[
\fint_{B(y_0, \kappa R)}w\le N\left(\kappa^\beta+\int_0^{\kappa R}
\frac{\tilde\omega_{\mathbf{A}}(t)}t \,dt \right) \fint_{B(y_0,R)}w,
\]
where $N$ is independent of $\kappa$.
We then fix $\kappa$ sufficiently small such that $2N\kappa^\beta\le \kappa^{\beta/2}$. Then for any small $R$ such that
\[
\int_0^{\kappa R}\frac{\tilde\omega_{\mathbf{A}}(t)}t \,dt\le \kappa^{\beta},
\]
we obtain
\[
\fint_{B(y_0, \kappa R)}w\le \kappa^{\beta/2} \fint_{B(y_0, R)}w.
\]
By iteration, we deduce $\fint_{B(y_0, r)} w\le Nr^{\beta/2}$.
This, however, contradicts with \cite[Theorem 1.5]{Es00}, which reads that for any $\varepsilon>0$, it holds that $\fint_{B(y_0, r)}w \gtrsim r^{\varepsilon}$ for all $r\in (0,1)$.
\end{proof}
\begin{acknowledgment}
Part of this work was done at the time when the second author was attending the Harmonic Analysis Program held at M.S.R.I. from January to May 2017. He would like to thank the members of the Institute and the organizers of the program for their hospitality.
\end{acknowledgment}

%
%

\end{document}